\documentclass[12pt]{amsart}
\usepackage{amsmath,amsfonts,amssymb,amsthm,amscd, latexsym, euscript, xypic, verbatim}
\usepackage[all]{xy}
\makeatletter
\def\section{\@startsection{section}{1}%
  \z@{1.1\linespacing\@plus\linespacing}{.8\linespacing}%
  {\normalfont\Large\scshape\centering}}
\makeatother

\theoremstyle{plain}

\newtheorem*{conj*}{Root Groups Conjecture}
\newtheorem*{thm1.2}{(1.2) Theorem}
\newtheorem*{thm1.3}{(1.3) Theorem}
\newtheorem*{thm1.4}{(1.4) Theorem}
\newtheorem*{prop*}{Proposition}

\newtheorem{prop}{Proposition}[section]
\newtheorem{question}[prop]{Question}
\newtheorem{thm}[prop]{Theorem}
\newtheorem{cor}[prop]{Corollary}
\newtheorem{lemma}[prop]{Lemma}

\theoremstyle{definition}
\newtheorem{Def}[prop]{Definition}
\newtheorem*{Def*}{Definition}
\newtheorem{Defs}[prop]{Definitions}

\newtheorem{example}[prop]{Example}

\newtheorem{notation}[prop]{Notation}
\newtheorem*{notation*}{Notation}
\newtheorem{remark}[prop]{Remark}
\newtheorem{remarks}[prop]{Remarks}

\newcommand{\zz}{\mathbb{Z}}

\newcommand{\frakG}{\mathfrak{G}}

\newcommand{\frakM}{\mathfrak{M}}
\newcommand{\frakN}{\mathfrak{N}}

\newcommand{\ga}{\alpha}
\newcommand{\gb}{\beta}
\newcommand{\gc}{\gamma}
\newcommand{\gC}{\Gamma}

\newcommand{\gd}{\delta}

\newcommand{\gl}{\lambda}

\newcommand{\gvp}{\varphi}
\newcommand{\gr}{\rho}
\newcommand{\gs}{\sigma}

\newcommand{\gt}{\tau}

\newcommand{\nsg}{\trianglelefteq}

\newcommand{\lr}{\longrightarrow}

\newcommand{\ra}{\rightarrow}

\newcommand{\sminus}{\smallsetminus}
\newcommand{\lan}{\langle}
\newcommand{\ran}{\rangle}

\newcommand{\Aut}{{\rm Aut}}

\newcommand{\half}{\textstyle{\frac{1}{2}}}

\newcommand{\widebar}[1]{\overset{\mskip1mu\hrulefill\mskip1mu}{#1}
                \vphantom{#1}}

\newcommand{\W}{\widebar}
\newcommand{\Wt}{\widetilde}
\newcommand{\cl}{\gC^{\gvp}}

\newcommand{\invlim}{\varprojlim}

\numberwithin{equation}{section}

\begin{document}
\title[Normal closure and normalizer of a homomorphism]{Normal closure and injective normalizer of a group homomorphism}
\author[Emmanuel D.~Farjoun,\quad Yoav Segev]
{Emmanuel~D.~Farjoun\qquad Yoav Segev}

\address{Emmanuel D.~Farjoun\\
        Department of Mathematics\\
        Hebrew University of Jerusalem, Givat Ram\\
        Jerusalem 91904\\
        Israel}
\email{farjoun@math.huji.ac.il}

\address{Yoav Segev\\
        Department of Mathematics\\
        Ben Gurion University\\
        Beer Sheva 84105\\
        Israel}
\email{yoavs@math.bgu.ac.il}

\begin{abstract}
Let $\gvp\colon\gC\to G$ be a homomorphism of groups.
We consider factorizations $\gC\xrightarrow{f} M\xrightarrow{g} G$
of $\gvp$ having  certain universal properties.
First we continue the investigation (see \cite{BHS}) of 
the case where $g$ is a universal normal map (our term for a crossed module).
Then we introduce and investigate a seemingly new {\em dual} case, where  $f$ is a universal normal map.
These two factorizations are natural generalizations
of the usual normal closure and normalizer of a subgroup.

Iteration of these universal factorizations yield 
certain towers associated to the map $\gvp$; we prove 
stability results for these towers.
In one of the cases we get a generalization of
the stability of
the automorphisms tower of a center-less group. 
The case where $g$ is a universal normal map 
is closely related  to hypercentral group extensions, 
Bousfield's localizations,  and the relative
Schur multiplier $H_2(G,\gC)=H_2(BG\cup_{B \gvp}{\rm Cone}(B\gC))$. 

Although our constructions here have 
strong ties to topological constructions
we  take here a group theoretical point of view.
\end{abstract}

\date{Version of November 2, 2014}
\maketitle

\section{Introduction and main results}\label{intro}

\noindent
Starting with  two standard constructions in group theory, 
namely the normal closure and the normalizer of a subgroup,  
we consider similar constructions for a general group homomorphism $\gvp\colon\gC\to G$.
We start with the
{\it free normal closure} of $\gvp$ (see below for its precise
relation to earlier works),  and  continue with the seemingly  new dual notion of {\it injective normalizer} of $\gvp$.

To settle the terminology,  we recall  the notion
of a {\it crossed module,} which in this paper we call {\it a normal map,} since  we
are trying to understand basic results about normal subgroups in the framework of general group maps.
Further motivation for the latter terminology was given in \cite{FS1} and comes from
topology: These maps have  a well-defined topological (or simplicial) group structure as homotopy cokernels or quotients $G//M$.

\begin{Def}\label{def normal map}
A {\it normal map} consists of a group homomorphism
\[
n\colon M\to G,
\]

together with an action of $G$ on $M$:
\[
\ell\colon G\to\Aut(M),
\]
which we call here {\it a normal structure} on $n$, such that  when denoting by $a^g$
the image of $a\in M$ under $\ell(g)$ for $g\in G$ (this notation
will prevail throughout this paper),
the following two requirement are satisfied.
\medskip

\begin{itemize}
  \item[(NM1)]  $(a^g)n=(an)^g$, for all $g\in G$ and $a\in M$.
\smallskip

  \item[(NM2)]  $a^{bn}=a^b$, for all $a, b\in M$.
\end{itemize}
{\it Note} that $a^g=a\ell(g),$ while $h^g=g^{-1}hg$ and $a^b=b^{-1}ab,$ for all $a, b\in M$ and
$h, g\in G$.  Note also that here we apply maps on the right.
\end{Def}

\noindent
Thus (see Lemma \ref{lem basic cm}) $M$ is a central extension of
the normal subgroup $n(M)\nsg G,$ coupled with a group action of $G$ (on $M$)
satisfying (NM1) and (NM2). 

The notion of a crossed module was introduced by J.~H.~C.~Whitehead (\cite{W1, W2, W3}).
He was motivated by attempts to capture the homotopy groups
of certain quotient spaces associated to a group homomorphism.
This notion is useful in many situations and has been
widely looked into,  
see, e.g., the book \cite{BHS}.

\setcounter{subsection}{1}
\subsection{An outline of the paper} 
In this paper we consider two decompositions via normal maps associated to a
given map of groups, one related to the usual  normal closure, the other to
the normalizer of a subgroup. We then study some of their properties and
consider
what happens upon repeating these constructions, proving some stability
results.

Let now $\gC\overset{\gvp}\ra G$ be a group homomorphism.  The {\it free normal closure} of $\gvp$ denoted
here by $\cl$ is a factorization
\begin{equation}\tag{FNC}\label{eq fnc}
\gC\overset{c_{\gvp}}\lr\cl\overset{\W\gvp}\lr G\qquad\quad(\gvp=c_{\gvp}\circ\W\gvp)
\end{equation}
of $\gvp,$ with $\W\gvp$ a normal map, having certain universal properties (see subsection \ref{sub nc} below).

The {\it injective normalizer} of $\gvp$  is a factorization
\begin{equation}\tag{IN}\label{eq in}
\gC\overset{\Wt\gvp}\lr N(\gvp)\overset{p_{\gvp}}\lr G\qquad\quad(\gvp=\Wt{\gvp}\circ p_{\gvp})
\end{equation}
with $\Wt\gvp$ a normal map, having certain universal properties (see subsection \ref{sub in} below).

As mentioned above  the ``free normal closure'' was
introduced and considered  in a  more
general setup: that of  {\it induced crossed module} as
in \cite{BH}.  In fact  if one takes $M=P$ in \cite[Definition 5.2.1, p.~109]{BHS},
then $f_*M$ is the present ``free normal closure" for the map $f$.   
Basic properties of the free normal closure
 were derived in \cite[Proposition 9 and 10]{BH} as well as in
\cite[Theorem 2.1]{BW1}, in chapter 5 of \cite{BHS} and in other papers.
We give the definition and the construction of the free normal closure, but
most of the details are deferred to Appendix \ref{app A}.  We need the basics of
the construction as we apply those in subsequent results, and to be  self contained.

The free normal closure and its repetitions have roots in topology. The first
was considered by  C.~Whitehead in his combinatorial homotopy work. The
tower of closures
  is related to the fundamental groups of  various approximations of spaces by 
nilpotent spaces as
in the work
of Bousfield and Kan \cite{BK}, the first-named author, and recently Goodwillie's
calculus tower \cite{G}. The decomposition we describe might allow for relative
versions of these well-known construction to be associated with maps of
spaces.

The notion of the {\it injective normalizer} is a dual notion.  It too has
strong topological background and analogues related to principal fibrations, to be considered elsewhere
\cite{F}.

We now briefly define the notions of the free normal closure
and of the injective normalizer.

\subsection{\bf The free normal closure of a group homomorphism}\label{sub nc}
Throughout this subsection let
\[
\gvp\colon \gC\to G,
\]
be a group homomorphism.  We associate to $\gvp$
a {\it factorization} as in equation \eqref{eq fnc}.
Furthermore $\W\gvp(\cl)=\lan\gvp(\gC)^G\ran,$ is the usual normal closure of
$\gvp(\gC)$ in $G$. Thus $\cl$ is a {\it central extension} of $\lan\gvp(\gC)^G\ran$,
coupled with a {\it group action of $G$} (on $\cl$) satisfying (NM1) and (NM2)
with respect to the map $n=\W\gvp$.

Moreover, the factorization $\gC \overset{c_{\gvp}}\lr\cl\overset{\W\gvp} \lr G$
is {\it universal}\, in the sense that  {\it any}\,   factorization
$\gC\overset{\psi}\lr M\overset{n}\lr G$ of $\gvp,$ with $n$
a normal map,
defines uniquely a normal morphism   $\cl\overset{\W{\psi}}\to M$ of normal maps
over  $G$ (see Definition \ref{def morphism of cm}) rendering the diagram
%
\begin{equation}\label{eq universality of nc}
\xymatrix{
&\gC\ar[ld]_{c_{\gvp}}\ar[rd]^{\psi}\\
\cl\ar[rd]_{\W{\gvp}}\ar[rr]^{\exists !\W{\psi}} && M\ar[dl]^{n}\\
&G\\
}
\end{equation}
%
commutative. In particular the {\it free normal closure is unique}.
The construction of $\cl$ is functorial
for the category of maps.
As an example we mention that if $\lan\gvp(\gC)^G\ran=G,$ then $\cl$ is just a central extension of $G$
together with a factorization as in equation \eqref{eq fnc}. 
In particular we prove (see Theorem \ref{thm center of fnc})
\setcounter{prop}{3}
\begin{thm}\label{thm surj fnc}
Suppose $\gvp\colon\gC\to G$ is a group homomorphism such
the normal closure $\lan\gvp(\gC)^G\ran=G$.  
Then the kernel of $\W\gvp$ is the relative homology
group $H_2(G,\gC)$ with respect to the map $\gvp$.
\end{thm}
In the case where $G\ne\lan\gvp(\gC)^G\ran,$ we ask

\begin{question}
Let $\gvp\colon\gC\to G$ be a group homomorphism.
What can be said about the structure of $\cl$?  What is the kernel of $\W\gvp\colon\cl\to G$?
\end{question}

Recall from \cite{CDFS} the notion of $A$-cellularity, for an arbitrary group $A$.
In Proposition \ref{prop cellular} we prove:

\begin{prop}
The group $\cl$ is $\gC$-cellular.
\end{prop}

\subsection*{\bf The free normal closures tower}\hfill
\medskip

\noindent
Notice that the process of taking the free normal closure can be iterated; this yields
the {\it (free) normal closures tower}: Let $\gvp_1:=\gvp, \gC_1=G$ and define inductively
$\gvp_{i+1}=c_{\gvp_i},$ and  $\gC_{i+1}=\gC^{\gvp_i}, i\ge 1$:

\begin{equation}\label{eq tower of fnc}
\xymatrix{
\gC\ar[dr]^>>>{\gvp_{k+1}}\ar@/^1ex/[drr]^{\gvp_k}\ar@/^2ex/[drrrr]^{\gvp_2}\ar@/^4ex/[drrrrr]^{\gvp_1}\\
\dots\ar[r]  &\gC_{k+1}\ar[r]^<<{\scriptscriptstyle{\W{\gvp_k}}} & 
\gC_k\ar[r] &
\dots \ar[r] & \gC_2\ar[r]^<<{{\scriptscriptstyle{\W{\gvp_1}}}} &G\\
\\
}
\end{equation}
Notice that diagram \eqref{eq tower of fnc} is commutative, the maps
$\W{\gvp_i}$ are normal maps and that $\gC_{i+1}$
is a central extension of the normal closure of $\gvp_i(\gC)$ in $\gC_i,$
for all $i\ge 1$.

Few points to note are:
\begin{itemize}
\item[(a)]
One can readily check (see Corollary \ref{cor eg}(1)) that if $\gvp$ is surjective then $\gC^{\gvp}=\gC/[\gC,\ker\gvp]$
and $c_{\gvp}\colon\gC\to \gC/[\gC,\ker\gvp]$ is the canonical homomorphism.
Thus if $G=1,$ then if we consider the normal closures tower we get
that $\gC_i=\gC/\gc_i(\gC),$ where $\gC=\gc_1(\gC)\ge\gc_2(\gC)\ge\dots$ is the
descending central series of $\gC$.
\end{itemize}
Thus the more challenging cases are when $\gvp$ is not surjective.
\begin{itemize}
\item[(b)]
In the case where  $\lan\gvp(\gC)^G\ran=G,$  {\it all}
the maps $\W{\gvp_i}$ are surjective, for integers $i\ge 1$
(see Lemma \ref{lem sfnc surj}(2)), so we get a series of central extensions
making diagram \eqref{eq tower of fnc} commutative. We prove (see Theorem \ref{thm series of fnc}):
\end{itemize}

\begin{thm}\label{thm tower of fnc}
Suppose that $\gC$ and $G$ are finite and that $G=\lan\gvp(\gC)^G\ran,$
then the normal closures  tower \eqref{eq tower of fnc} terminates
after a finite number of steps.
\end{thm}

Note now  that Example \ref{eg inj abelian} shows that if $\gC$ and $G$
are non-trivial finite abelian groups and $\gvp$ is {\it not} surjective,
then the size of the (finite abelian) groups $\gC_i$ of diagram \eqref{eq tower of fnc} grows
to infinity.   However we ask:

\begin{question}\footnotemark\label{ques inv limit}
Suppose that $\gC$ and $G$ are finite.  Is it true that the inverse limit
$\gC_\infty :=\varprojlim\gC_i,$ where $\gC_i$ are as in diagram \eqref{eq tower of fnc}, is finite?
\end{question}
\footnotetext{The answer to Question \ref{ques inv limit} is given in \cite{FS2}.}
It is interesting to note the behavior of the normal closures tower on abelianiziations.
In  Proposition \ref{prop inj abelianizations} we prove:
 
\begin{prop}
Let $\gC_\infty :=\invlim\gC_i,$ and let $\gvp_{\infty}\colon\gC\to\gC_{\infty}$
be the map obtained by the universal property of $\gC_{\infty}$. Then,
\begin{enumerate}
\item
the map $(c_{\gvp})_{ab}\colon \gC_{ab}\to\cl_{ab}$ induced by $c_{\gvp}$
is injective;

\item
the map $(\gvp_{\infty})_{ab}\colon\gC_{ab}\to (\gC_{\infty})_{ab}$ induced
by $\gvp_{\infty}$ is injective.
\end{enumerate}
\end{prop}

Going back to question \ref{ques inv limit}, it is reasonable to expect that the normal closures tower is pro-equivalent to a fixed finite
group that gives the {\em universal subnormal factorization} $\gC\xrightarrow{\gvp_{\infty}} \gC_{\infty}\to G,$  of the original map.  
For a general group map (not necessarily of finite groups),  
we should get a  relative version  of  the nilpotent and the Bousfield completion
of a group $\gC$ that is closely related to the tower of fundamental 
groups of a topological (``relative nilpotent'') completion tower.

The existence and uniqueness of the free normal closure
are  recalled in \S\ref{sect free nc}  and Appendix \ref{app A}.
Some of its  properties are given in \S\ref{sub some properties}.
We note already at this early stage that if $\gC\nsg G$ and $\gvp$
is inclusion, we {\it do not} always get that $\cl=\gC$ (see Example \ref{eg inj abelian}).

\setcounter{subsection}{9}
\subsection{{\bf The injective normalizer of a group homomorphism}}\label{sub in}
Here we add  a construction which, in some sense, is ``dual'' to
the construction of the free normal closure. Namely with every group map $\gvp\colon\gC\to G$
we associate a factorization as in equation \eqref{eq in}.
Further, this factorization is {\it injective}
in the sense that  any   factorization $\gC\to H\to G$ of $\gvp$ with $\gC\to H$
a normal map
defines uniquely a normal morphism  $H\to N(\gvp)$.
In particular the injective normalizer is unique.
In this case the construction is funcotrial in the variable $G,$
assuming $\gC$ is fixed.

The image $p_{\gvp}(N(\gvp))$ is always a subgroup of the normalizer
$N_G(\gvp(\gC)),$ but is not always equal to it (see Lemma \ref{lem useful}(1) and Remark \ref{rem triv. map}(3)).
As opposed to the free normal closure, $N(\gvp)$  {\it does  agree} with the usual normalizer $N_G(\gvp(\gC))$
if $\gvp$ is injective.  

As in the case of the free normal closure we can iterate the process
of taking the injective normalizer and we obtain the {\it (injective) normalizers tower}

\begin{equation}\label{eq inj nor tower}
\xymatrix{
\gC^0\ar[r]^{\Wt{\gvp_0}}\ar[rrrd]_{\gvp_0} & \gC^1\ar[r]^{\Wt{\gvp_1}}\ar[rrd]^{\gvp_1} & \gC^2\ar[r]\ar[rd]
& \gC^3\ar[r]\ar[d] &\dots
\ar[r] & \gC^{\ga}\ar[r]^{\Wt{\gvp_{\ga}}}\ar[lld]_{\gvp_{\ga}}& \dots\\
& & &  G
}
\end{equation}

\noindent
where here $\gC^0=\gC$ and $\gvp_0=\gvp$.  Further if $\ga\ge 1$ is not a limit ordinal,
then $\gC^{\ga}:=N(\gvp_{\ga-1})$ and
$\gvp_{\ga}:=p_{\gvp_{\ga-1}}$.
If $\ga$ is a limit ordinal then we take the obvious direct limit:
$\gC^{\ga}:=\varinjlim_{\gb<\ga}\gC^{\gb},$ and $\gvp_{\ga}\colon\gC^{\ga}\to G$
is the map obtained from the universal property of the direct limit.
Note that $\Wt{\gvp_{\ga}}$ is a normal map for every ordinal $\ga$ which is not a limit ordinal.

We note that if $\gvp\colon \gC\to 1,$ then $N(\gvp)=\Aut(\gC),$ and
$p_{\gvp}\colon N(\gvp)\to 1$.  Thus the   normalizers tower
is the automorphism tower of the group $G$.  We prove  (see Theorem \ref{thm nor tower})
a relative version of the well known stability result of the automorphism tower of a finite group.

\setcounter{prop}{10}
\begin{thm}
Let $\gvp\colon\gC\to G$ be a group homomorphism.  If $\, \gC$ and $G$ are finite
and $Z(\ker\gvp)=1$,
then the  normalizers tower
terminates after a finite number of steps.
\end{thm} 

It turns out that if $\gC\to G$ is inclusion and $G$ is finite, then the normalizers
tower is just the tower of the usual normalizers which of course stops.
Recall that by \cite{Ha},  the automorphism tower of any group terminates 
(see also more recent related works of J.~D.~Hamkins).
One expects that the transfinite normalizer tower terminates as well, for an arbitrary map $\gvp,$
in the spirit of Shelah \cite{KS}.

The following remark indicates that one can detect whether $\gvp\colon\gC\to G$ is
a normal map using the injective normalizer $N(\gvp)$:

\begin{remark}\label{rem retract}
The map $\gvp\colon\gC\to G$ is a normal map iff $\gvp$ is a retract of $\Wt{\gvp},$
i.e., there exists a section $s\colon G\to N(\gvp)$ such that the following diagram
\[
\xymatrix{
\gC\ar[r]^{=}\ar[d]_{\gvp} &\gC\ar[r]^{=}\ar[d]^{\Wt{\gvp}} &\gC\ar[d]^{\gvp}\\
G\ar[r]^s & N(\gvp)\ar[r]^{p_{\gvp}} & G\\
}
\]
is commutative (so $s\circ p_{\gvp}={\rm id}_{G}$). See Lemma \ref{lem nmap}.
\end{remark}

We conclude the introduction with a remark putting our work in a more general framework.

\begin{remark}\label{rem adjoint}
We note that one can view the two constructions in this paper as functors 
adjoint to
the corresponding forgetful functors. 
However, we give  and use here   explicit constructions of  these two adjoint functors. 
These constructions are  the main tools
used to demonstrate some of the properties of these adjoint functors.

Let $\frakN\frakM$  (resp.~$\frakN\frakM^{\gC}$) be the category of normal maps (resp.~normal  maps from a {\it fixed group}  $\gC$)  
of groups, and $\frakG^2$  (resp.~$\frakG^{\gC}$)  be the category of maps 
of groups (resp.~maps from a {\it fixed group} $\gC$).

Consider the the forgetful (``Underlying'') functor  to the category of group maps:
\[
U\colon \frakN\frakM\to \frakG^2.
\]
It is not hard to see that it commutes with inverse limits, but does not commute in general  with direct limits.
However the  restriction   $U^\gC$ of  $U$  to  $\frakN\frakM^{\gC}$    
does commutes with direct limits. Thus one expects that $U$
has a left adjoint and that $U^\gC$ has a right adjoint. The left adjoint of $U$ namely  $cl: \frakG^2\to \frakN\frakM$ 
is called here the {\it free normal closure} and is denoted by $(\gvp\colon\gC\to G) \mapsto  (\cl\to G)$.

The right adjoint of $U^{\gC}$ namely
${\rm nor}^{\gC}: \frakG^{\gC} \to \frakN\frakM^{\gC}$
is call here the {\it injective normalizer} 
and is denoted by $(\gvp\colon\gC\to G) \mapsto  (\gC\to N(\gvp))$.

The factorizations of $\gC\to G$ by the normalizer and the normal closure arise as the natural augmentation
of these functors.

Now for these  two functors one can consider as usual ``algebras'' namely objects which are
retract of composition  $U\circ F$ where $F$ is one of the above adjoint functors to $U$.
It turns out (see Remark \ref{rem retract} above), 
that a retract of the ${\rm nor}^{\gC}$ is exactly a normal map namely such a retraction  exactly equips  a map with a normal structure. 
It is not clear what ``algebra''  is given by  a retract of $U\circ cl.$
\end{remark}

\section{Preliminaries: Normal maps (Crossed modules)}\label{sect cm}

Recall from Definition \ref{def normal map} the notion of a normal map.

\begin{lemma}\label{lem basic cm}
Let $n\colon M\to G$ be a normal map.  Then
\begin{enumerate}
\item
$\ker(n)\le Z(M),$ and $\ker(n)$ is a $G$-invariant subgroup of $M$;

\item
$n(M)\nsg G$ and the map $n\colon M\to n(M)$ is a normal map;
\item
if $N\le M$ is a $G$-invariant subgroup of $M,$ then the 
restriction $n\colon N\to G$ is a normal map with the same normal
structure, restricted to $N$;

\item
if $V$ is an abelian group, then the map $($also denoted $n$$)$
$n\colon M\times V \to G,$ defined by $(a,v)n=an,$
is a normal map with the normal structure $(a,v)^g=(a^g,v),$ for all
$a\in M, v\in V$ and $g\in G$.
\end{enumerate}
\end{lemma} 
\begin{proof}
Part (1) is well known: if $b\in\ker n,$ then, by (NM2), $a^b=a^{bn}=a,$ for all
$a\in M,$ so $b\in Z(M)$. Also, if $a\in\ker(n),$ then by (NM1),
$(a^g)n=(an)^g=1,$ so $a^g\in\ker(n)$.   Part (2) is also well known:
By (NM1) we have $(n(a))^g=n(a^g),$ for all $a\in M$ and $g\in G$.
It is also clear that the second part of (2) holds.
Part (3) is obvious, simply observe that (NM1) and (NM2) hold
with $M$ replaced by $N$.

For part (4) we check that
\[
((a,v)^g)n=(a^g,v)n=(a^g)n=(an)^g=((a,v)n)^g,
\]
so (NM1) holds.  Also,
\[
(a,v)^{((b,w)n)}=(a,v)^{bn}=(a^{bn},v)=(a^b,v)=(a,v)^{(b,w)}
\]
so (NM2) holds as well.
\end{proof}

\begin{remark}\label{rem surj}
Let $n\colon M\to G$ be a surjective map such that $\ker(n)\le Z(M)$.
Then there is a natural action of $G$ on $M,$ where $a^g=a^b,$
with $a\in M,$ $g\in G$ and $b\in M$ is an element such that $n(b)=g$.  It is easy 
to check that this definition is independent of the choice of $b,$ and 
that $n$ becomes a normal map over $G$.  In fact the above is the unique normal structure on $n$.
\end{remark}

We require the following well established notions of morphisms between normal maps.

\begin{Def}\label{def morphism of cm}
Let $n_i\colon M_i\to G_i,$ $i=1, 2,$ be two normal maps.
A {\it normal morphism} from $n_1$ to $n_2$ is a pair of maps
$(\mu, \eta)$ such that the diagram
\[
\xymatrix{
M_1\ar[r]^{\mu}\ar[d]_{n_1} & M_2\ar[d]^{n_2}\\
G_1\ar[r]^{\eta} &G_2\\
}
\] 
commutes, and such that
$\mu(m_1^{g_1})=(\mu(m_1))^{\eta(g_1)},$ for all $m_1\in M_1$ and $g_1\in G_1$.
If $M_1=M_2,$ we always assume that $\mu$ is the identity map,
and if $G_1=G_2,$ we always assume that $\eta$ is the identity map.
\end{Def}

\begin{lemma}\label{lem ab}
Suppose we are given a commutative diagram
\begin{equation}\label{diag preab}
\xymatrix{
\gC\ar[rr]^{\gvp}\ar[dd]_{\mu}&& G\ar[dd]^{\eta}\\
& M'\ar[rd]^{n'}\\
\gC'\ar[rr]^{\gvp'}\ar[ru]^{\psi'} && G'
}
\end{equation}
with $n'$ a normal map, and let 

\[
\xymatrix{
M\ar[r]^{n}\ar[d]_{\pi_2} & G\ar[d]^{\eta}\\
M'\ar[r]^{n'} & G'
}
\]
be a pull back diagram.

Then $n$ is a normal map having the normal structure $(m',h)^g=((m')^{g\eta},h^g),$
for all $(m',h)\in M$ and all $g\in G$.  Furthermore, $\pi_2$ is a normal
morphism and there is a map $\psi\colon\gC\to M$
such that the diagram
\begin{equation}\label{diag ab}
\xymatrix{
& M\ar[dd]_{\pi_2}\ar[rrd]^n\\
\gC\ar@/_3ex/[rrr]_{\gvp}\ar[dd]_{\mu}\ar[ur]^{\psi}&&& G\ar[dd]^{\eta}\\
& M'\ar[rrd]^{n'}\\
\gC'\ar[rrr]^{\gvp'}\ar[ru]^{\psi'} &&& G'
}
\end{equation}
commutes.
\end{lemma}
\begin{proof}
We have
\[
((m',h)^g)n=((m')^{g\eta},h^g)n=h^g=((m',h)n)^g,
\]
for all $(m', h)\in M$ and $g\in G,$ so (NM1) holds for $n$.  Also
\[
(m',h)^{(a',g)n}=(m',h)^g=((m')^{g\eta},h^g)=((m')^{(a')n'},h^g)=((m')^{a'},h^g)=(m',h)^{(a',g)},
\]
for all $(m',h), (a',g)\in M,$ so (NM2) holds for $n$ as well, 
and $n$ is a normal map with the given normal structure.
It is easy to check that $\pi_2$ is a normal morphism.

Let $\psi\colon\gC\to M$ be defined by
\[
\gc \psi=((\gc)\mu \psi', (\gc)\gvp).
\]
Then, by definition, $\gc \psi n=\gc\gvp,$ for all $\gc\in\gC,$  
so $\psi\circ n=\gvp$.  Also, by definition, $\psi\circ\pi_2=\mu\circ \psi'$.
\end{proof}

\section{The free normal closure of a map}\label{sect free nc}

\noindent
In this section $\gvp\colon\gC\to G$ is a fixed map.
We recall that   
the {\it free normal closure of $\gvp$} 
is a factorization $\gC\overset{c_{\gvp}}\lr\cl\overset{\W\gvp}\lr G$ of $\gvp,$
with $\W\gvp$ a normal map, as defined below.

\begin{Def}\label{def normal closure}
Let $\gvp\colon \gC\to G$ be a map.
A {free normal closure} of $\gvp$ is a factorization of the latter via
\[
\gC\overset {c_{\gvp}}\lr \cl\overset{\W\gvp}\to G
\] 
such that $\W\gvp\colon\cl\to G$ is a normal map, and such that for any other factorization
\[
\gvp=\psi\circ n
\]
via a
normal map $n\colon M\to G$ 
with $\psi\colon \gC\to M,$
there exists a unique normal morphism  $\W\psi\colon\cl\to M$ rendering
the diagram below commutative.
%
%
\begin{equation}\label{diag fnc}
\xymatrix{
\gC\ar[rr]^{\gvp}\ar[rd]^{c_{\gvp}}\ar@/_3ex/[rdd]_{\psi} && G\\
&\cl\ar[ru]^{\W\gvp}\ar[d]_{\exists!\W\psi}\\
&M\ar[ruu]_n\\
}
\end{equation}
\end{Def}
In \S2 below, and Appendix \ref{app A}, we will show that the free normal closure exists.  As for
uniqueness we have:

\begin{lemma}\label{uniqueness of nc}
The free normal closure of $\gvp\colon\gC\to G$ is unique
up to an isomorphism of normal maps over $G$.
\end{lemma}
\begin{proof}
Straightforward from the universal properties.
\end{proof}

\begin{lemma}\label{lem gen of WgC}
Let $\gC\overset{c_{\gvp}}\lr\cl\overset{\W\gvp}\lr G$  be
the free normal closure of $\gvp$.  Then
\begin{enumerate}
\item 
the group $\cl$    
is generated by $\{(c_{\gvp}(\gC))^g\mid g\in G\}$;

\item
$\W\gvp(\cl)$ is the normal closure of the subgroup $\gvp(\gC)$ in $G$.
\end{enumerate}
\end{lemma}
\begin{proof}
Let $M\le\cl$ be the subgroup generated by $\{(c_{\gvp}(\gC))^g\mid g\in G\},$
and consider diagram \eqref{diag fnc}, where $n$ is the map $\W\gvp$ restricted to $M$,
and where $\psi$ is the map $c_{\gvp}$ (with range $M$ in place of $\cl$). 
By Lemma \ref{lem basic cm}(3), $m$ is a normal morphism.  Hence there exists
a (unique) normal morphism 
\begin{equation}\label{eq hatc}
\W\psi\colon\cl\to M,
\end{equation}
rendering the diagram commutative.

Now consider again diagram \eqref{diag fnc}, with $\cl, \W\gvp, c_{\gvp}$ in place of $M, m, \psi$ respectively.
Of course $\W{c_{\gvp}}$ is the identity map in this case.  However, the map $\W\psi$ of equation
\eqref{eq hatc} considered as a map from $\cl$ to $\cl$ also
renders diagram \eqref{diag fnc} commutative in this case.  By uniqueness,
$\W\psi$ is the identity map, so part (1) holds.

Next, by (1), $\W\gvp(\cl)$ is generated by
$\{\W\gvp(c_{\gvp}((\gC)^g))\mid g\in G\}$, since $\W\gvp$ is a group homomorphism.
But since $\W\gvp$ is a normal map, $\W\gvp(c_{\gvp}(\gC)^g)=(\W\gvp(c_{\gvp}(\gC))^g=\gvp(\gC)^g,$
for all $g\in G$.  Hence (2) holds.
\end{proof}

\setcounter{subsection}{3}
\subsection{A construction of the free normal closure}\label{sect ex nc}\hfill
\medskip

\noindent

The purpose of this subsection is to recall the construction
of the free  normal closure of a map $\gvp\colon\gC\to G$.
The detailed proofs are given in Appendix \ref{app A}.
 
\setcounter{prop}{4}
\begin{thm}\label{thm normal closure}
Let $\gvp\colon\gC\to G$ be a map of groups.  Then the free normal
closure of $\gvp$ exists.
\end{thm}
We start by considering the free group $F$ generated by the
following set of distinct symbols:
\[
\gC_G:=\{\gc_g\mid \gc\in\gC,\ g\in G\}.
\]
We consider the following relations on $F$
%
%
\begin{equation}\label{eq 1g=1h}\tag{$R_1$}
R_1:=\{1_g=1_h\mid  g,h\in G\}.
\end{equation}
%
%
\begin{equation}\label{eq Rnm1}\tag{$R_{nm1}$}
R_{nm1}:=\{\gc_g\gd_g=(\gc\gd)_g\mid \gc,\gd\in\gC\text{ and }g\in G\}.
\end{equation}
%
%
\begin{equation}\label{eq Rnm2}\tag{$R_{nm2}$}
R_{nm2}:=\{\gd_h^{-1}\gc_g\gd_h=\gc_{g\gvp(\gd)^h}\mid\gc,\gd\in\gC\text{ and }h,g\in G\}.
\end{equation}
%
The relation $R_1$ just mean that the identity $1=1_g,$ where $g\in G$ is
an arbitrary element, is the identity of $F$, and $F$ is the free group
on the set 
\[
(\gC\sminus\{1\})_G:=\{\gc_g\mid\gc\in\gC\sminus\{1\}\text{ and } g\in G\}.  
\]
We let $\cl$ be the group defined using the relations $R_1, R_{nm1}$ and $R_{nm2}$
above:
%
\begin{equation}\label{eq hatgC}
\cl:={\rm Gr}\{\gC_G\mid R_1,\ R_{nm1},\ R_{nm2}\}
\end{equation}
%
\begin{notation}\label{not widebarF}
\begin{enumerate}
\item
We denote 
\[
\widehat F:={\rm Gr}\{\gC_G\mid 1_h=1_g,\ \gc_g\gd_g=(\gc\gd)_g\mid \gc,\gd\in\gC\text{ and }g,h \in G\}.
\]
and we let
$\widehat\gC_g \text{ be the image of the set }\{\gc_g\mid\gc\in\gC\}\subseteq F\text{ in }\widehat F.$
Thus $\widehat{\gC}_g=\{\widehat{\gc}_g\mid g\in G\},$ and  
$\widehat{\gC}_g\cong \gC,$ for all $g\in G$.  Further, $\widehat F$ is a free product 
\[
\widehat F=\star_{g\in G}\widehat\gC_g.
\]

\item
Let $g\in G$.  We denote by $\W\gc_g$ 
the image in $\cl$ of $\gc_g\in F$.  We let 
\begin{equation}\label{not widehatgC}
\W\gC_g:=\{\W\gc_g\mid \gc\in\gC\text{ and }g\in G\}\le\cl.
\end{equation}
\end{enumerate}
\end{notation}
\noindent
We define
\begin{equation}\label{eq hatgvp}
\W\gvp\colon \cl\to G:\quad \W\gc_g\mapsto \gvp(\gc)^g,\quad\gc\in\gC,\ g\in G.
\end{equation}
Lemma \ref{lem hatgvp respects R} shows that $\W\gvp$ is a well defined map.

Next we define an action $\W\ell\colon G\to \Aut(\cl)$ by
\begin{equation}\label{eq action on cl}
\W\ell(h)\in\Aut(\cl):\quad (\W\gc_g)^h\mapsto\W\gc_{gh},\qquad\text{for all $\gc\in\gC$ and $g, h\in G.$}
\end{equation}

\noindent
Lemma \ref{lem the aut h} shows that $\W\ell$ defines an action of $G$ on $\cl,$ and
Lemma \ref{lem Wgvp is a nm} shows that $\W\gvp\colon\cl\to G$ is a normal map having
$\W\ell$ as its normal structure.  Finally, in subsection \ref{sub univ of nc}
we show the universality property of $\cl$.  

\setcounter{subsection}{6}
\subsection{Examples of finiteness and cellularity}\label{sub some properties}\hfill
\medskip

\noindent
In this subsection we prove some basic properties of the normal closure 
and give some examples.

\setcounter{prop}{7}
\begin{lemma}\label{lem widehatgC}
Let $h,g\in G,$ and $\gd\in\gC$. Then
\begin{enumerate}
\item
$\W\gC_g^{\W\gd_h}=\W\gC_{g\gvp(\gd)^h};$
\smallskip

\item
$\W\gC_g=\W\gC_{\gvp(\gd) g};$
\smallskip

\item
$\ker\W\gvp\le Z(\cl)$;

\item
if $G$ is finite, then
$\cl=\Pi_{g\in G}\W\gC_g$.
\end{enumerate}
\end{lemma} 
\begin{proof}
Recall the notation $\W\gC_g$ from equation \eqref{not widehatgC}.
Since $\cl$ satisfies the relations $R_{nm2}$ (see equation \eqref{eq Rnm2}), 
we see that (1) holds.  Also, by the relations $R_{nm2},$
\[
\W\gd_g^{-1}\W\gc_g\W\gd_g=\W\gc_{g\gvp(\gd)^g}=\W\gc_{\gvp(\gd)g}.
\]
It follows that
\[
\W\gC_g=\W\gC_g^{\W\gd_g}=\W\gC_{\gvp(\gd)g}.
\]
This shows (2), and (3) follows from Lemma \ref{lem basic cm}(1), since
$\W\gvp$ is a normal map.

For the proof of part (4) set $G:=\{g_1,\dots, g_s\},$ where $s=|G|$.
Recall that $\cl$ is the image of the free product
$\widehat F=\widehat\gC_{g_1}*\widehat\gC_{g_2}*\cdots  *\widehat\gC_{g_s},$
where $\widehat\gC_{g_i}\cong\gC,$ for all $i$ (see Notation \ref{not widebarF}).
Of course $\widehat F$ is equipped with a natural free product
word length.  For $\W w\in\cl$, we let $|\W w|$ be the minimal
length of a word in $\widehat w\in\widehat F$ such that $\W w$ is the
image of $\widehat w$.  We now show by induction on $|\W w|,$ that 
$\W w\in \W\gC_{g_1}\W\gC_{g_2}\cdots\W\gC_{g_s}$.

If $|\W w|=1,$ this is obvious.
Our induction hypothesis is that if $|\W w| < r,$ then $r\le s+1,$ and we can
write $\W w\in \W\gC_{g_1}\W\gC_{g_2}\cdots\W\gC_{g_s}$ using $|\W w|$ 
non-identity elements. 

Assume $|\W w|=r$.   
From all the words $\widehat w\in\widehat F$ of length $r$ whose image
is $\W w,$ choose a word 
\[
\widehat w=(\widehat\gc_1)_{g_{i_1}}(\widehat\gc_2)_{g_{i_2}}\cdots(\widehat\gc_r)_{g_{i_r}}
\]  
so that $i_1$ is as small as possible.
By induction, we may assume
that $i_2< i_3 <\cdots < i_r$.  Now if $i_2=i_1,$ then
$\widehat\gc_{g_{i_1}}\widehat\gc_{g_{i_2}}\in\widehat\gC_{g_{i_1}},$ so
we get that $|\W w| < r,$ a contradiction.  If $i_2< i_1,$
then, using the relations $R_{nm2},$ we can write
$\W w$ as a word of length $r$ starting with $(\W\gc_{i_2})_{g_{i_2}},$
contradicting the minimality of $i_1$.  Hence $i_1<i_2,$ and (4) holds.      
\end{proof}

Part (2) of the following corollary should be compared with \cite[Theorem 5.7.1, p.124]{BHS}.

\begin{cor}\label{cor eg}
Let $\gvp\colon\gC\to G$ be a map of groups and set $K:=\ker\gvp,$ then
\begin{enumerate}
\item
if $\gvp$ is surjective, the $\cl=\W\gC_1,$ and
$\cl\cong \gC/[\gC,K]$ as crossed modules over $G$;   

\item
If $\gC$ and $G$ are finite, then $\cl$ is finite.  
\end{enumerate}
\end{cor}
\begin{proof}
(1)\quad  
Since $\gvp$ is surjective $\cl=\W\gC_1$ by Lemma \ref{lem widehatgC}(2).

Since $\gC/[\gC,K]$ is a central extension of $G,$ Remark \ref{rem surj}
shows that the map $n\colon \gC/[\gC,K]\to G$ induced by $\gvp$ is a normal map.
Further the canonical map
$\psi\colon\gC\to \gC/[\gC,K]$ satisfies $\psi(\gc)=1$ iff
$\gc\in[\gC,K]$ and $\psi\circ n=\gvp$. 

Let $\W\psi\colon\cl\to\gC/[\gC,K]$ be the unique map of diagram \eqref{diag fnc}. 
The map $c_{\gvp}\colon\gC\to\cl$ is surjective
and satisfies $c_{\gvp}\circ\W\psi=\psi$.  Thus $\W\psi(\W\gc_1)=\psi(\gc)$.
It follows that if $1=c_{\gvp}(\gc)=\W\gc_1,$ then $\psi(\gc)=1,$ so $\gc\in [\gC,K]$.
Thus $\ker c_{\gvp}\le [\gC, K]$.
On the other hand,
$c_{\gvp}(K)\le\ker\W\gvp,$
so, since $\W\gvp$ is a normal map, $c_{\gvp}(K)\le Z(\cl)$.
It follows that $[\gC,K]\le\ker c_{\gvp}$.  Thus $\ker c_{\gvp}=[\gC,K]$.
Since $\W\psi(\W\gc_1)=\psi(\gc),$ we see that $\W\psi$
is a normal isomorphism from $\cl$ to $\gC/[\gC,K]$.  
\smallskip

\noindent
(2)\quad  
This follows immediately from Lemma \ref{lem widehatgC}(4).
\end{proof}

Our next proposition shows that $\cl$ is $\gC$-cellular (see \cite{CDFS} for the notion of cellularity).

\begin{prop}\label{prop cellular}
$\cl$ is $\gC$-cellular.
\end{prop}
\begin{proof}
We show that $\cl$ is the  {\it coequalizer} of two maps between two free products
of copies of $\gC$:
\[
\xymatrix{\underset{i\in I}\star\gC_i\ar@<1ex>[rr]^{e_1}\ar@<-.5ex>[rr]_{e_2}&& \underset{g\in G}\star\widehat\gC_g,\\
}
\]
where $\widehat\gC_g\cong\gC$ is as in notation \ref{not widebarF}, for $g\in G$.
Also, $I=\gC\times G\times G,$ and $\gC_i\cong\gC,$ for $i\in I$. 
We now define  the maps $e_1$ and $e_2$.  Let $i=(\gd; g, h)\in I$.  For $\gc\in\gC_i,$ let
$e_1(\gc)=\widehat{\gc}_{g\gvp(\gd)^h},$ this defines the homomorphism $e_1$.
To define $e_2,$ let $i=(\gd; g, h)\in I$ and for $\gc\in\gC_i,$ define
$e_2(\gc)=(\widehat{\gd}_h)^{^{-1}}\widehat{\gc}_g\widehat{\gd}_h$. 
This defines the homomorphism $e_2$.  By the construction of $\cl,$
and by the definition of the coequilzer, $\cl$ is the coequilizer
of these maps, hence $\cl$ is $\gC$-cellular.
\end{proof}

%
\setcounter{subsection}{10}
\subsection{On abelian quotients of the normal closure}\hfill
\medskip

\noindent
The following results show that the  normal
closures tower behaves well with respect to abelianization.
We denote $H_{ab}=H/[H,H]$ the abelianization of $H,$ for any group $H$.
%
\begin{prop}\label{prop inj abelianizations} 
Let $\{\gC_i\mid i=1, 2, 3,\dots\}$
be the normal closures tower of $\gvp$ (see diagram \eqref{eq tower of fnc}).  Let  $\gC_{\infty}=\invlim\gC_i,$
and let $\gvp_{\infty}\colon\gC\to\gC_{\infty}$ be the map obtained by the universal property of $\gC_{\infty},$ then
\begin{enumerate}
\item
the map $(c_{\gvp})_{ab}\colon \gC_{ab}\to\cl_{ab}$ induced by $c_{\gvp}$
is injective; hence,
\item
the map $(\gvp_{\infty})_{ab}\colon\gC_{ab}\to (\gC_{\infty})_{ab}$ induced
on abelianization is injective.
\end{enumerate} 
\end{prop}
\begin{proof}
Consider diagram \eqref{diag preab} with $\gC'=\gC_{ab}, G'=G_{ab},$
$\mu$ and $\eta$ are the natural maps and $\gvp'=\gvp_{ab}$
is the map induced by $\gvp$.  Further let $M'=\gC_{ab}^{^{\gvp_{ab}}},$
let $\psi'=c_{\gvp_{ab}}$ and $n'=\W{\gvp_{ab}}$.

By Lemma \ref{lem ab}, the pullback $M$  renders diagram 
\eqref{diag ab} commutative.  Thus by the universal property
of $\cl,$ there exists a normal morphism $\W{\psi}\colon\cl\to M$
rendering diagram \eqref{diag fnc} commutative.  We thus get a commutative diagram

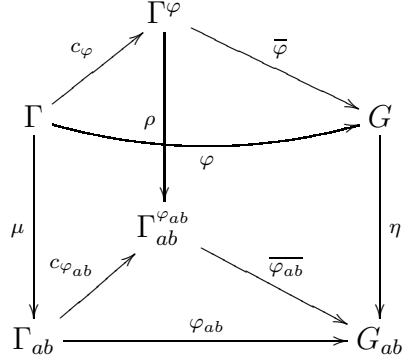
\begin{equation}\label{diag abab}
\xymatrix{
& \cl\ar[dd]_{\gr}\ar[rrd]^{\W{\gvp}}\\
\gC\ar@/_2ex/[rrr]_{\gvp}\ar[dd]_{\mu}\ar[ur]^{c_{\gvp}}&&& G\ar[dd]^{\eta}\\
& \gC_{ab}^{^{\gvp_{ab}}}\ar[rrd]^{\W{\gvp_{ab}}}\\
\gC_{ab}\ar[rrr]^{\gvp_{ab}}\ar[ru]^{c_{\gvp_{ab}}} &&& G_{ab}
}
\end{equation}
By Example \ref{eg inj abelian}, $\gC_{ab}^{^{\gvp_{ab}}}$ is abelian
and $c_{\gvp_{ab}}$ is injective.

Hence we have a commutative diagram (note that $\mu$ is surjective)
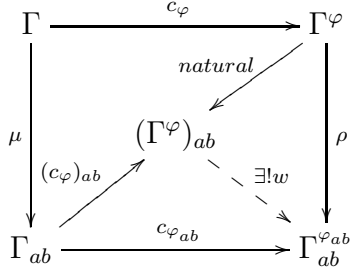
\[
\xymatrix{\gC\ar[rr]^{c_{\gvp}}\ar[dd]_{\mu}&& \cl\ar[dd]^{\gr}\ar[dl]_{\tiny natural}\\
& (\cl)_{ab}\ar@{-->}[rd]^{\exists!w}\\
\gC_{ab}\ar[rr]^{c_{\gvp_{ab}}}\ar[ru]^{(c_{\gvp})_{ab}}&&\gC_{ab}^{^{\gvp_{ab}}}\\
}
\]
Since $c_{\gvp_{ab}}$ is injective, so is $(c_{\gvp})_{ab}$.
This shows (1).  Then (2) follows from (1), since by the universality
property of $\gC_{\infty}$ there is a map $\gvp_{\infty}\colon\gC\to\gC_{\infty},$
such that $\gvp_{\infty}\circ\psi_2=\gvp_2,$ where $\psi_2\colon\gC_{\infty}\to\gC_2$
is the canonical map.  Hence $(\gvp_{\infty})_{ab}\circ(\psi_2)_{ab}=(\gvp_2)_{ab}$. Since $\gvp_2=c_{\gvp},$ 
and by (1), $(c_{\gvp})_{ab}$ is injective,  (2) follows.
\end{proof}

\section{Stability of the normal closures tower}
\noindent
The purpose of this section is to prove:

\begin{thm}\label{thm series of fnc}
Let $\gC$ and $G$ be finite groups and let $\gvp\colon\gC\to G$ be a homomorphism.
Assume that $G=\lan\gvp(\gC)^G\ran$.  Then the  normal closures tower 
corresponding to $\gvp$ terminates after a finite number of steps.
Furthermore, the last term of the normal closures tower has size less or
equal $|\gC|\cdot f(|G|),$ where $f$ is defined in equation \eqref{eq f} below.
\end{thm}

We first make a general observation about the  normal closures tower.

\begin{lemma}\label{lem sfnc surj}
Let $\gvp\colon\gC\to G$ be a group homomorphism.  Suppose $G=\lan\gvp(\gC)^G\ran,$
then
\begin{enumerate}
\item 
$\gC_i=\lan\gvp_i(\gC)^{\gC_i}\ran,$ for all integers $i\ge 1;$ 

\item
$\gC_{i+1}$ is a central extension of $\gC_i,$ for all $i\ge 1$.
\end{enumerate}
where $\gC_i$ are
the terms of the normal closures tower as in diagram \eqref{eq tower of fnc}.
\end{lemma}
\begin{proof}
By Lemma \ref{lem gen of WgC}(2) the image of $\W{\gvp_i}$ in $\gC_i$
is $\lan\gvp_i(\gC)^{\gC_i}\ran$.  By hypothesis,  $\W{\gvp_1}$ is surjective,
that is, $\gC_2$ is a central extension of $\gC_1=G$.  Now let $i\ge 2,$ and suppose
that $\gC_i$ is a central extension of $\gC_{i-1}$.  We show that $\gC_{i+1}$
is a central extension of $\gC_i$.   Indeed, by Lemma \ref{lem gen of WgC}(1),
$\gC_i=\gC^{\gvp_i}$    
is generated by $\{(\gvp_i(\gC))^g\mid g\in \gC_{i-1}\}$.  However, since $\gC_i$
is a central extension of $\gC_{i-1},$ this just means that (1) holds.  Thus
$\W{\gvp_i}$ is surjective.  Since $\W{\gvp_i}$ is a normal map, (2) holds.
\end{proof}

For Theorem \ref{thm hypercentral} below, let us recall
that the upper central series of  any group $M$ is the ascending series
\[
1= Z_0(M)\le Z_1(M)\le\dots Z_{\ga}(M)\le Z_{\ga+1}(M)\le\dots Z_{\gd}(M)=Z_{\infty}(M),
\]
given by  $Z_1(M) = Z(M)$ is the center of $M$, and recursively by
$Z_{\ga+1}(M)/Z_{\ga}(M)=Z(M/Z_{\ga}(M))$ for all ordinals $\ga,$ and $Z_{\gl}(M)=\bigcup_{\mu<\gl}Z_{\mu}$
for every limit ordinal $\gl$. The last term $Z_{\infty}(M)$ of this series 
is called the {\it hyper-center} of $M$.  The group $M$ is {\it hyper-central} if $Z_{\infty}(M)=M$.

\begin{thm}[see Theorem B, p.~2598 in \cite{KOS}]\label{thm hypercentral}
There exists an integer valued function $f$ such that if  $M$ is a group satisfying
$|M/Z_{\infty}(M)|=t<\infty,$ 
then  $M$ contains a finite normal subgroup $L,$ with $|L|\le f(t)$ and such
that $M/L$ is hyper-central  Here
\begin{equation}\label{eq f}
f(t)=t^k,\quad\text{where}\quad k=\half(log_pt + 1)\text{ and $p$ is the least prime divisor of $t$}.
\end{equation} 
\end{thm}

\begin{lemma}\label{lem normal closure nil}
Let $N$ be a nilpotent group. Suppose $T\le N$ is
a subgroup such that $N=\lan T^N\ran$.  Then $N=T$.
\end{lemma}
\begin{proof}
The Frattini factor group $N/\Phi(N)$ is abelian, so
since  $N=\lan T^N\ran,$ we get that $N=T\Phi(N),$
so $N=T$.
\end{proof}

\begin{lemma}\label{lem M}
Let $\gC$ and $G$ be finite groups.  Assume that $M$
is a finite group such that 
\begin{itemize}
\item[(a)]
$M/Z_{\infty}(M)$ is isomorphic to a quotient of $G.$

\item[(b)]
There exists a homomorphism $c\colon\gC\to M$ such that $M=\lan c(\gC)^M\ran$.
\end{itemize}
Then $|M|\le |\gC|\cdot f(|G|),$ where $f$ is as in equation \eqref{eq f}.
\end{lemma}
\begin{proof}
By Theorem \ref{thm hypercentral} and by (a) we can find  $L\nsg M$ such that $|L|\le f(|G|)$ and such $M/L$ is hyper-central
(note that $f(|G/K|)\le f(|G|),$ for any  $K\nsg G$).
Of course $M/L$ is nilpotent, as $M/L$ is finite.  Since by (b) the normal
closure of the image of $c(\gC)$ in $M/L$ is $M/L$,   Lemma 
\ref{lem normal closure nil} implies that $M/L$ is equal to that image.  In 
particular $|M/L|\le |\gC|$.   This proves the lemma.
\end{proof}

\begin{proof}[Proof of Theorem \ref{thm series of fnc}]
Let $i\ge 1$. We apply Lemma \ref{lem M} with $M=\gC_i$ and with $c=\gvp_i$.
By Lemma \ref{lem sfnc surj}(1), hypothesis (b) of Lemma \ref{lem M} holds, and
by Lemma \ref{lem sfnc surj}(2), hypothesis (a) of Lemma \ref{lem M} holds.
Hence, $|\gC_i|\le |\gC|\cdot f(|G|)$.  This of course proves the Theorem.
\end{proof}

\section{Central $f$-extensions and relative Schur multiplier}\label{sub f-central}
\noindent
The purpose of this section is to prove Theorem \ref{thm surj fnc} 
of the introduction. Hence we assume that $\gvp\colon\gC\to G$ satisfies
$G=\lan\gvp(\gC)^G\ran$.  
By Lemma \ref{lem gen of WgC}(2), the map $\W{\gvp}\colon\cl\to G$ is surjective,
so $\cl$ is a central extension of $G$.  Further, for any factorization
$\gC\xrightarrow{\psi} M\xrightarrow{n} G$ of $\gvp,$ with $n$ a normal
map, $M$ is a central extension of $G$.  This is because by Lemma \ref{lem basic cm}(2),
$n(M)$ is normal in $G,$ so $n$ is surjective since $\gvp(\gC)\le n(M)$. Also by Lemma \ref{lem basic cm}(1), $\ker(n)\le Z(M)$.
Thus $\cl$ is {\it universal} amongst all central extensions $M$ of $G$
such that $\gvp$ factors through $M\to G$.  Indeed, for any such $M$
there is a unique normal morphism $\W{\psi}\colon\cl\to M$ rendering
diagram \eqref{diag fnc} commutative.

To identify $\ker\gvp$ we show that $(\cl,c_{\gvp})$ is a {\it universal
central $\gvp$-extension} in a sense to be made precise shortly.
The detailed account of this identification will appear in \cite{FS3}.
Here we only give the basic definitions and the main results of \cite{FS3}.
(We note that \cite{FS3} proves a more general result, see 
Proposition \ref{prop f-central} and Theorem \ref{thm f-central} below.) 
 
Let us consider central extensions
\[
0\to A\to M\to G\to 1.
\]
of $G$.  Let us also {\it fix a map}
\[
f\colon\gC\to G
\]
The following are the basic concepts used in this section.

\begin{Defs}\label{def f-central}
\begin{enumerate}
\item
A central $f$-extension of $G$  is a pair $(M,\psi),$ where $M$ is a central extension of $G$
with kernel $A,$
{\it together with} a map $\psi\colon\gC\to M$ that factorizes  $f$
as in diagram \eqref{diag f-central} below.

\begin{equation}\label{diag f-central}
\xymatrix{
&&&\Gamma\ar[dl]_{\psi}\ar[dr]^f\\
0\ar[r]&A\ar[r]&M\ar[rr]^n&&G\ar[r]& 0\\
}
\end{equation}
\item
A map between two central $f$-extensions $(M,\psi)$ and $(M',\psi')$ of $G$ with kernels $A,A'$ respectively,
 is a map of the underlying extensions
which is the identity on $G,$ as in the commutative diagram \eqref{diag map f-central}
below.
$M$ and $M'$  are called {\it equivalent} if $\gt$ below is an isomorphism and $\kappa$ below is the
identity.

\begin{equation}\label{diag map f-central}
\xymatrix{
0\ar[r] &A\ar[r]\ar[dd]_{\kappa} & M\ar[rr]^n\ar[dd]_{\tau} &&G\ar[r]\ar[dd]^{=} & 1\\
&&&\ar[ru]^f\Gamma\ar[rd]^f\ar[lu]_{\psi}\ar[ld]_{\psi'}\\
0\ar[r]&A'\ar[r]&M'\ar[rr]^{n'} && G\ar[r]&1\\
}
\end{equation}
\end{enumerate}
\end{Defs} 
 
In the following results, 
for an abelian group $A$, by $H_*(G,\gC; A)$ we mean\linebreak 
 $H_*(BG\cup_{B f}{\rm Cone}(B\gC); A)$, and similarly for relative cohomology.

\begin{prop}[\cite{FS3}]\label{prop f-central}
The equivalence classes of central $f$-extension of $G$ with a given  kernel $A$ 
have a natural abelian group structure, and are classified by
the relative cohomology group $H^2(G,\gC; A)$, with coefficients $A$.
\end{prop}
 
Proposition \ref{prop f-central} yields the following theorem:

\begin{thm}[\cite{FS3}]\label{thm f-central}
Assume that the map $f_{ab}\colon \gC_{ab}\to G_{ab}$ induced on the abelianizations is
surjective.  Then
there exists a universal  central $f$-extension
$(U,\eta)$ of $G$ with kernel  $H_2(G,\gC;\zz)$.
For any central $f$-extension  $(E,\psi)$ of $G$ there
is a unique map of central $f$-extensions (as in Definition \ref{def f-central}(2))
from $U$ to $M$.
\end{thm}
 
As an immediate corollary to Theorem \ref{thm f-central} we get

\begin{thm}\label{thm center of fnc}
Let $\gvp\colon\gC\to G$ be a group homomorphism.  
Assume that $G=\lan\gvp(\gC)^G\ran$.  Then
$\cl$  is the universal central $\gvp$-extension of $G$ of Theorem \ref{thm f-central}.
In particular the kernel of $\W\gvp$ is $H_2(G,\gC;\zz)$.
\end{thm}
\begin{proof}
The first thing to notice is that $\gvp_{ab}\colon\gC_{ab}\to G_{ab}$ is surjective since $G=\lan\gvp(\gC)^G\ran$. 
Further, we already noted (see Lemma \ref{lem gen of WgC}(2)) that if $G=\lan\gvp(\gC)^G\ran,$ then $\cl$
is a central extension of $G$.  Also, as noted in the beginning of
this section, the fact that $(M,\psi)$ is a central $\gvp$-extension 
of $G$ is equivalent to a factorization
$\gC\xrightarrow{\psi} M\xrightarrow{n} G$ of $\gvp$ with $n$ a 
normal map.  Thus the universal property that defines $\cl$
is precisely the universal property that defines the universal
central $\gvp$-extension (of Theorem \ref{thm f-central}).  Hence these are isomorphic, and the remaining
part of the theorem follows from Theorem \ref{thm f-central}.
\end{proof}

\begin{remark}
The  construction in \cite{FS3} of the universal central $f$-extension $(U,\eta)$
of Theorem \ref{thm f-central} 
extends the   Schur  universal
central extension of a perfect group $G$ (the case $\gC=1$), to any map between groups
(not necessarily perfect groups) $f\colon \gC\to G,$ inducing surjection on abelianizations.
\end{remark}

\section{Further examples of the free normal closure}\label{sec eg nc}
\noindent
Throughout this section $\gvp\colon\gC\to G$ is a map of groups and 
$\gC\overset{c_{\gvp}}\to\cl\overset{\W\gvp}\to G$ is
its universal normal closure.  Further, throughout this section,
$\gC\overset{\psi}\to M\overset{n}\to G$ is a factorization
of $\gvp$ (i.e.~$\psi\circ n=\gvp$) with $n$ a normal map,
as in diagram \eqref{diag fnc}, and $\W\psi\colon\cl\to M$ is the unique normal
morphism as in diagram \eqref{diag fnc}.

As we saw in Corollary \ref{cor eg},
if $\gvp$ is surjective, then there is a normal isomorphism $\cl\to\gC/[\gC,\ker\gvp]$.
We also saw that if $\gC$ and $G$ are finite, then $\cl$ is finite.
We consider here some more examples, some of which appear in the litrature.
We give details since we need those as a starting point for some of the results above.
A  lemma for the case where $\gvp$
is injective and $\gvp(\gC)$ is normal in $G$
is given in Appendix \ref{app B}.

\begin{example}\label{decending central series}
Suppose $G=1$.  Let $\gvp_1:=\gvp$.
Then, by Corollary \ref{cor eg}(2), the free normal closure of $\gvp$
is the factorization $\gC\overset{c_{\gvp_1}}\to \gC/[\gC,\gC]\to 1$.  Since $\gvp_2:=c_{\gvp_1}$
is the canonical map it is
a surjective map whose kernel is $[\gC,\gC]$, Corollary \ref{cor eg}(2) applies again
and shows that the free normal closer
of $\gvp_2$ is the factorization $\gC\overset{\gvp_3}\lr \gC/\gc_3(\gC)\overset{\W{\gvp_2}}\lr \gC/\gc_2(\gC)$.
Proceeding
in this way we see that the normal closures tower of diagram 
\eqref{eq tower of fnc}
are the quotients $\gC_i=\gC/\gc_i(\gC),$ $i\ge 1,$
where  $\gc_i(\gC)$ are the members of its descending central series
$\gC=\gc_1(\gC)\ge\gc_2(\gC)\ge\dots$.
\end{example}

The following example is well known.

\begin{example}\label{eg inj abelian}
Assume that $\gC$ and $G$ are abelian groups.
Let $\gvp\colon \gC\to  G$ be a homomorphism.  We claim that
\begin{equation}\label{eq gvp inj to abelian}
\cl=\bigoplus_{x\in G/\gvp(\gC)}\gC_x\qquad\text{and}\qquad c_{\gvp}(\gc)=\gc_{\gvp(\gC)},\quad \forall \gc\in \gC.
\end{equation}
where $\gC_x=\{\gc_x\mid \gc\in\gC\}\cong \gC,$ for  $x\in G/\gvp(\gC)$.
The action of $G$
on $\cl$ is given by $\gc_x^g=\gc_{xg},$ for all $x\in G/\gvp(\gC)$ and  $g\in G$.
The map $\W{\gvp}$ takes $\gc_{\gvp(\gc)},$  to $\gc,$ for all $\gc\in\gC$
and takes $\gC_x$ to $0,$ for all $x\ne\gvp(\gC)$. 

To see that equation \eqref{eq gvp inj to abelian} is correct,
note the relation  \eqref{eq Rnm2}, gives $(\W{\gd}_g)^{^{-1}}\W{\gc}_g\W{\gd}_g=\gc_{g\gvp(\gd)}$,
for all $\W{\gd}_g, \W{\gc}_g\in\W{\gC}_g$   (taking $h=g$, see the notation in equation \eqref{not widehatgC}),
because $\gC$ is abelian. So since $\W{\gC}_g$ is abelain we see that $\W{\gc}_g=\W{\gc}_{g\gvp(\gd)}$.
But now using relation \eqref{eq Rnm2} again we see that
\[
(\W{\gd}_h)^{^{-1}}\W{\gc}_g\W{\gd}_h=\W{\gc}_{g\gvp(\gd)}=\W{\gc}_g,
\]
for all $\W{\gd}_h\in\W{\gC}_h$ and  $\W{\gc}_g\in\W{\gC}_g,$ where $g, h\in G,$
again because $G$ is abelian.  Since $\cl$ is generated by $\{\W{\gC}_g\mid g\in G\}$
we see that $\cl$ is abelian.

Next, if $\gC\xrightarrow{\psi} M\xrightarrow{n} G$
is a factorization of $\gvp,$ with $n$ a normal map,
then $M$ is a central extension of $\gvp(\gC),$ so
$M$ is nilpotent of class at most $2$.  Since the normal
closure of $\psi(\gC)$ in $M$ is contained in $\psi(\gC)Z(M),$
and since $\psi(\gC)$ is abelian, we see that $\lan\psi(\gC)^M\ran$
is abelian. Further,
$\psi(\gC)^{(\gc)\gvp}=\psi(\gC)^{(\gc)\psi n}=\psi(\gC)^{(\gc)\psi}=\psi(\gC),$
for all $\gc\in\gC,$ so we see that $\psi(\gC)^h=\psi(\gC)^g,$
for all $x\in G/\gvp(\gC)$ and all $g, h\in x$. 

Thus the map from $\W{\psi}\colon \cl\to M$ taking $\gc_x$ to $\psi(\gc)^h,$ where $\gc\in\gC, x\in G/\gvp(\gC)$
and $h\in x$ is a well defined map making diagram \eqref{diag fnc} commutative.
It is now routine to check that $\bigoplus_{x\in G/\gvp(\gC)}\gC_x$
is the free normal closure of $\gvp$.  
\end{example}

\begin{example}
There are many examples where $\gvp$ is injective, but
$\cl$ is {\it not} the normal closure of $\gvp(\gC)$.
Here is one.

Let $G=A_5={\rm PSL}_2(5),$ and
Let $\gC\le G$ be a subgroup of order $3$ (thus $\gvp$ is the inclusion map).
Then the normal closure of $\gC$ in $G$ is $G$ (since
$G$ is simple).  Now let $M={\rm SL}_2(5)$ be the universal perfect central extension of $G$.
Let $n\colon M\to G$
be the natural surjection, and let $\psi\colon \gC\to M,$
such that $\psi(\gC)$ is a subgroup of order $3$ in $M$
and such that $\psi\circ n=inc$ is the inclusion map $inc\colon\gC\to G$.

Consider diagram \eqref{diag fnc}.  Since the only homomorphism $A_5\to {\rm SL}_2(5)$ is the
trivial map, if $\cl=G,$ then $\W\psi$ would be the trivial map.
But $c_{\gvp}\circ\W\psi= \psi$ which means that $\psi$ is the trivial map,
a contradiction. 

It is interesting to note that $\cl$ in this case is ${\rm SL}_2(5)\times \zz_3,$
and $c_{\gvp}(\gC)$ is not contained in ${\rm SL}_2(5)$.  Thus the center
of $\cl$ is $\zz_2\times\zz_3=\zz_6$.
\end{example}

\begin{example}
Assume that $\gC$ and $G$ are  finite groups.
Let $C$ be the normal closure of $\gvp(\gC)$ in $G$.
Assume  that $\gC$ and  $C$ are $\pi$-groups.  
Recall that a $\pi$-group, for a set of primes $\pi,$
is a group $H$ such that any prime $p$ dividing
the order of $H$ is in $\pi$.
 
We claim that $\cl$
is also a finite $\pi$-group, further, if $C$ is solvable (nilpotent), so is $\cl$.
If $C$ and $\gC$ are $p$-groups so is $\cl$.
  
By Corollary \ref{cor eg}(2),
$\cl$ is finite.
By Lemma \ref{lem widehatgC}(3),
$\ker\W\gvp\le Z(\cl),$ so by Lemma \ref{lem gen of WgC}(2), since the
normal closure of $\gvp(\gC)$ in $G$ is a $\pi$-group, we see that
$\cl/\ker\W\gvp$ is a $\pi$-group.  Let $\pi'$ be the complement to $\pi$ in the set
of all primes.  If $O_{\pi'}(Z(\cl))\ne 1,$ then,
by the Schur-Zassenhaus Theorem (\cite[(18.1), p.~70]{A}),
we would get $\cl=O_{\pi}(\cl)\times O_{\pi'}(\cl)$.
But then $c_{\gvp}(\gC)\le O_{\pi}(\cl)$, since $c_{\gvp}(\gC)$ is a $\pi$-subgroup of $\cl$.
Hence also $\cl=\lan c_{\gvp}(\gC)^g\mid g\in G\ran\le O_{\pi}(\cl)$.
But since $\cl$ is generated by $\{c_{\gvp}(\gC)^g\mid g\in G\},$ we see that
$O_{\pi'}(\cl)=1$.
\end{example}

\section{The injective normalizer of a map}\label{sect inj nor}

\begin{Def}\label{def injective normalizer}
Let $\gvp\colon \gC\to G$ be a map of groups.
The {\it injective normalizer} of $\gvp$ is a factorization of the latter via
\[
\gC\overset{\Wt\gvp}\lr N(\gvp)\overset{p_{\gvp}}\lr G
\] 
such that $\Wt\gvp\colon\gC\to N(\gvp)$ is a normal map, and such that for any other factorization
$\gvp=n\circ f$
via a
normal map $n\colon \gC\to H$ 
with $f\colon H\to G,$
there exists a unique  normal morphism  $\Wt{f}\colon H\to N(\gvp)$ rendering
the diagram below commutative.

\begin{equation}\label{diag in}
\xymatrix{
\gC\ar[rr]^{\gvp}\ar[rd]^{n}\ar@/_2ex/[rdd]_{\Wt\gvp} && G\\
&H\ar[ru]^{f}\ar[d]_{\exists!\Wt{f}}\\
&N(\gvp)\ar[ruu]_{p_{\gvp}}\\
}
\end{equation}
\end{Def}
In subsection \ref{sect ex inj nor} below we will show that the injective normalizer exists.  We have

\begin{lemma}\label{uniqueness of in}
The injective normalizer of $\gvp\colon\gC\to G$ is unique
up to a normal isomorphism.
\end{lemma}
\begin{proof}
Straightforward from the universal properties.
\end{proof}

\setcounter{subsection}{2}
\subsection{The construction of the injective normalizer}\label{sect ex inj nor}\hfill
\medskip

\noindent
Throughout this subsection let $\gvp\colon\gC\to G$ be a map of groups.

\setcounter{prop}{3}
\begin{Def}
Let $\gt\in\Aut(\gC)$ and $g\in G$.  We say that
$\gt$ and $g$ are {\it compatible} if the following
diagram is commutative:
\[
\xymatrix{
\gC\ar[r]^{\gvp}\ar[d]_{\gt}& G\ar[d]^{c_g}\\
\gC\ar[r]^{\gvp} & G
}
\]
where $c_g$ is conjugation by $G$.  Thus
\begin{equation}\label{eq compatible}
\gvp(\gc)^g=\gvp(\gt(\gc)),\qquad\forall\gc\in\gC.
\end{equation}
\end{Def}

\begin{Def}\label{def inj nor}
The {\it injective normalizer} of $\gvp$ is the subgroup of $\Aut(\gC)\times N_G(\gvp(\gC))$ given by
\[
N(\gvp):=\{(\gt,g)\mid (\gt,g)\text{ is a compatible pair}\}.
\]
\end{Def}

The next lemma shows that $N(\gvp)$ is indeed a subgroup of $\Aut(\gC)\times N_G(\gvp(\gC))$.
 
\begin{lemma}\label{lem Ngvp is a gp}
\begin{enumerate}
\item
$(id, 1)$ is a compatible pair;

\item
if $(\gt, g)$ is a compatible pair, then $(\gt^{-1},g^{-1})$ 
is a compatible pair.

\item 
the product of two compatible pairs is compatible;
\end{enumerate}
\end{lemma}
\begin{proof}
(1):\quad
We take in equation \eqref{eq compatible} $\gt=id$ and $g=1$ and we get
\[
\gvp(\gc)^1=\gvp(id(\gc)),\qquad\forall\gc\in\gC.
\]
Hence (1) holds.
\medskip

\noindent
(2):\quad
Since $(\gt, g)$ is compatible we have
\[
\gvp(\gc)^g=\gvp(\gt(\gc)),\qquad\forall\gc\in\gC.
\]
Now replace each $\gc$ with $\gt^{-1}(\gc)$ to get
\[
\gvp(\gt^{-1}(\gc))^g=\gvp(\gc),\qquad\forall\gc\in\gC,
\]
and now conjugate both sides with $g^{-1}$ to get
\[
\gvp(\gc)^{g^{-1}}=\gvp(\gt^{-1}(\gc)),\qquad\forall\gc\in\gC,
\]
This shows (2).
\medskip

\noindent
(3):\quad
We have
\[
\gvp(\gc)^v=\gvp(\gr(\gc)),\qquad\forall\gc\in\gC,
\]
for $(\gr, v)=(\gt, g)$ and $(\gs, h)$.
Thus
\[
(\gvp(\gc))^{gh}=\gvp((\gc)\gt)^h=\gvp((\gc)\gt\gs),\qquad,\forall\gc\in\gC.
\]
Therefore $(\gt\gs, gh)$ is a compatible pair.
\end{proof}

\begin{remarks}\label{rem triv. map}
\begin{enumerate}
\item
If $\gvp$ is the trivial map,
then any pair $(\gt, g)\in \Aut(\gC)\times G$ is a compatible pair,
so $N(\gvp)=\Aut(\gC)\times G$.
\item
If $\gvp$ is inclusion, then $N(\gvp)=\{\gvp c_g\gvp^{-1}, g)\mid g\in N_G(\gC)\},$ where $c_g$
is the inner automorphism of $G$ induced by $g$.  Hence $N(\gvp)\cong N_G(\gC)$.

\item
Of course $p_{\gvp}(N(\gvp))$ does not always equal to $N_G(\gvp(\gC))$.
For example, let $\gC$ be an extra-special group of order $2^{2n+1}, n\ge 2$ which is a central product
of n dihedral groups of order $8$, and let
$G$ be a semi-direct product $E\rtimes L,$ where $E$ is an elementary abelian group of order $2^n,$
and $L\cong {\rm GL}(n,2)$.  Note that $\Aut(\gC)$ is an extension
of an elementary abelian group of order $2^n$ by the orthogonal group $O_+(2n,2)$.
Let  $\gvp\colon\gC\to G,$ be the natural map taking
$\gC$ onto $E$.
Then not every $g\in L$ lifts to an automorphism of $\gC,$ so $p_{\gvp}(N(\gvp))\ne G=N_G(\gvp(\gC))$. 
\end{enumerate}
\end{remarks}

\begin{lemma}\label{lem tildegvp and p}
Let
\[
\Wt\gvp\colon\gC\to N(\gvp):\quad \gc\mapsto (c_{\gc},\gvp(\gc)),\quad\forall\gc\in\gC.
\]
Let $N(\gvp)$ act on $\gC$ via $\gc^{(\gt, g)}=(\gc)\gt,$ for all $\gc\in\gC$.
Let
\[
p_{\gvp}\colon N(\gvp)\to G:\quad (\gt, g)\mapsto g,
\]
be the projection on the second coordinate.  Then
$\Wt\gvp$ is a normal map, and $\Wt\gvp\circ p_{\gvp}=\gvp$.  
\end{lemma}
\begin{proof}
First note that $\Wt\gvp$ is a group map from $\gC$ to $N(\gvp)$.
We check (NM1).  For all $\gc\in\gC$ and $(\gt,g)\in N(\gvp),$ we have
\[
(\gc^{(\gt, g)})\Wt\gvp=(\gc)\gt\Wt\gvp=(c_{\gc\gt},(\gc)\gt\gvp).
\]
while
\[
((\gc)\Wt\gvp)^{(\gt,g)}=(c_{\gc}, (\gc)\gvp)^{(\gt,g)}=(c_{\gc}^{\gt},((\gc)\gvp)^g)=(c_{\gc\gt},(\gc)\gt\gvp),
\]
where the last equality follows from the fact that $(\gt, g)$ is a compatible pair.
This shows (NM1).

We now check (NM2):
\[
(\gc)^{(\gd)\Wt\gvp}=(\gc)^{(c_{\gd},\gvp(\gd))}=(\gc)c_{\gd}=\gc^{\gd},\quad\forall\gc, \gd\in\gC.
\]
This shows (NM2) and that $\Wt\gvp$ is a normal map.
Clearly $\Wt\gvp\circ p_{\gvp}=\gvp$.
\end{proof}
 
\begin{prop}\label{prop N(gvp)}
Let $\gC\overset{n}\to H\overset{f}\to G$ be a decomposition
of $\gvp,$ with $n$ a normal map.  Then there exists a
unique map $\Wt{f}\colon H\to N(\gvp)$  rendering diagram \eqref{diag in} commutative.
\end{prop} 
\begin{proof}
For each $h\in H$ we let $\ell_h\in\Aut(\gC)$ be
the automorphism that comes from the normal structure
of $n$. As usual, we write $\ell_h(\gc)=\gc^h,$ for all
$\gc\in\gC$ and $h\in H$.  Let 
\[
\Wt{f}\colon H\to N(\gvp):\quad h\mapsto (\ell_h, f(h)).
\]
We first claim that $(\ell_h, f(h)$ is a compatible pair.
By equation \eqref{eq compatible} we must check that
\[
\gvp(\gc)^{f(h)}=\gvp(\ell_h(\gc)),\qquad\forall\gc\in\gC.
\]
Now
\[
\gvp(\gc^h)=f(n(\gc^h))=f(n(\gc)^h)=(f(n(\gc)))^{f(h)}=\gvp(\gc)^{f(h)},
\]
for all $h\in H,$ thus $\Wt{f}$ is indeed a map from $H$ to $N(\gvp)$. 
Clearly $\Wt{f}$ is a group map.  Also
\[
(\gc)n\circ \Wt{f}=\Wt{f}(n(\gc))=(\ell_{n(\gc)}, f(n(\gc)))\overset{(*)}=(c_{\gc},\gvp(\gc))=\Wt\gvp(\gc),
\] 
for all $\gc\in\gC,$ where the equality $(*)$ follows from the fact that $n$ is a 
normal map.

Next
\[
(h)\Wt{f}\circ p_{\gvp}=(\ell_h, f(h))p_{\gvp}=f(h).
\]
Thus $\Wt{f}$ renders diagram \eqref{diag in} commutative.

Finally, we must show that
\[
\gc^h=\gc^{\Wt{f}(h)},\qquad\text{ for all }\gc\in\gC\text{ and }h\in H.
\]
However, by definition 
\[
\gc^{\Wt{f}(h)}=\gc^{(\ell_h, f(h))}=(\gc)\ell_h=\gc^h.
\]

It remains to show that $\Wt{f}$ is unique.  So let 
\[
\widehat f\colon H\to N(\gvp),
\]
rendering diagram \eqref{diag in} commutative.  Let $h\in H$ and let
$\widehat{f}(h)=(\gt, g)$.  Then $f(h)=p_{\gvp}(\widehat{f}(h))=g,$ so $g=f(h)$.
Also, since $\widehat{f}$ is a normal morphism,  
\[
\gc^h=\gc^{\widehat{f}(h)}=(\gc)\gt,
\]
so $\gt=\ell_h,$ and so $\Wt{f}$ is unique and the 
proof of the proposition is complete.
\end{proof}

\begin{remark}
Let $\psi\colon G\to G',$ and consider the diagram
\[
\xymatrix{
\gC\ar[rr]^{\gvp}\ar[rd]^{\Wt\gvp}\ar@/_3ex/[rdd]_{\Wt{\gvp\circ\psi}} && G\ar[r]^{\psi} &G'\\
&N(\gvp)\ar[ru]^{p}\ar[d]_{\Wt{p\circ \psi}}\\
&N(\gvp\circ\psi)\ar[rruu]_{\pi}\\
}
\]
Thus we see that $\psi$ induces a map $\Wt{p\circ\psi}\colon N(\gvp)\to N(\gvp\circ \psi)$.
\end{remark}

Our next lemma shows that the injective normalizer of $\gvp$ precisely  detects
whether $\gvp$ is a normal map.

\begin{lemma}\label{lem nmap}
The map $\gvp\colon\gC\to G$ is a normal map iff the map
$p_{\gvp}\colon N(\gvp)\to G$ has a section $s\colon G\to N(\gvp)$
such that $\gvp\circ s=\Wt{\gvp}$.
\end{lemma} 
\begin{proof}
Assume first that $\gvp$ is a normal map.
Then $\gvp$ decomposes as $\gC\xrightarrow{\gvp} G\xrightarrow{id} G,$
with $\gvp$ a normal map.  Thus $H, f$ of diagram \eqref{diag in}
are $G,  id$ respectively. The universality property of the injective normalizer
implies the existence of a map $s\colon G\to N(\gvp)$ (the map $\Wt{f}$)
rendering diagram \ref{diag in} commutative.  In particular $s\circ p_{\gvp}=id,$
and $\gvp\circ s=\Wt{\gvp}$.

Conversely, assume there exists $s\colon G\to N(\gvp)$ with $s\circ p_{\gvp}=id$ and $\gvp\circ s=\Wt{\gvp}$.
Define an action of $G$ on $\gC$ as follows.  Let $\pi\colon N(\gvp)\to\Aut(\gC)$ be the 
projection on the first coordinate, and let $\gc^g=\pi(s(g))(\gc),$ for all $\gc\in\gC$
and $g\in G$.  We show that this is a normal structure on $\gvp$.

Let $\gc\in\gC$ and $g\in G$.  Then $(\gc^g)\gvp=\gvp(\pi(s(g))(\gc))$.
Notice now that $s(g)=(\gt,g)\in N(\gvp),$ with $\gt\in\Aut(\gC),$ since $s\circ p_{\gvp}=id$.
Thus  $(\gc^g)\gvp=\gvp(\pi(s(g))(\gc))=\gvp(\gt(\gc))$.  Since $(\gt, g)$ is a compatible
pair, $\gvp(\gt(\gc))=\gvp(\gc)^g$.  Thus we see that $(\gc^g)\gvp=\gvp(\gc)^g$ and (NM1) holds.

Next for $\gc, \gd\in\gC$ we have 
\[
(\gc)^{\gd\gvp}=\pi(s(\gvp(\gd)))(\gc)=\pi(\Wt{\gvp}(\gd))(\gc)=c_{\gd}(\gc)=\gc^{\gd},
\]
and (NM2) holds.
\end{proof}
\section{Stability of the normalizers tower}
\noindent
In this section $\gvp\colon\gC\to G$ is a homomorphism
of groups and $\gC\overset{\Wt\gvp}\lr N(\gvp)\overset{p_{\gvp}}\lr G$ is
its injective normalizer.  Theorem \ref{thm nor tower} below generalizes the obvious
observation that the repeated normalizer of any subgroup of a finite group
stabilizes (the case where $\gvp$ is injective), and the well known
stability of the automorphism tower of a finite group with trivial center
(here: the case where $G=1$, see Remark \ref{rem triv. map}(1)). 

The following result will be required for the proof of
Theorem \ref{thm nor tower}.
 
\begin{thm}\label{thm W}
\begin{enumerate}
\item
$($Wielandt, \cite[13.5.2]{R}$)$ Let $H$ be a finite group and let
$K$ be a subnormal subgroup of $H$ such that $C_H(K)=1$.
Then there is an upper bound on $|H|$ depending only on $|K|$.

\item
$($\cite[13.5.3]{R}$)$ Let $\ga$ be any ordinal, and
let $K=K_0\nsg K_1\nsg\dots\nsg K_{\ga}\nsg K_{\ga+1}\nsg\dots$ be an 
ascending chain of groups such that $C_{K_{\ga+1}}(K_{\ga})=1$, for all $\ga$.
Then  $C_{K_{\ga}}(K_0)=1$.
\end{enumerate}
\end{thm}

\begin{lemma}\label{lem useful}
We have
\begin{enumerate}
\item
$p_{\gvp}(N(\gvp))\le N_G(\gvp(\gC))$;

\item
if $(\gt, g)\in N(\gvp),$ then $\gt(\ker\gvp)=\ker\gvp,$ and hence 
$\gt$ acts on $\gC/\ker\gvp$;

\item
$C_{N(\gvp)}(\Wt\gvp(\gC))=\{(\gt, g)\in N(\gvp)\mid \gt\in C_{\Aut(\gC)}(\gC/Z(\gC))\text{ and }g\in C_G(\gvp(\gC))\}$;

\item
$\ker p_{\gvp}=\{(\gt, 1)\in N(\gvp)\mid \gt\in C_{\Aut(\gC)}(\gC/\ker\gvp)\}$ and $\Wt{\gvp}(\ker\gvp)\le\ker p_{\gvp}$;

\item
$\ker\Wt\gvp=Z(\gC)\cap\ker\gvp$;

\item
if $\ker\Wt\gvp=1,$ then  $\ker\Wt{\gvp_1}=1$, where $\gvp_1:=p_{\gvp}\colon N(\gvp)\to G$;

\item
set $p_{\gvp}=\gvp_1$ and assume that  $Z(\ker\gvp)=1,$ then
\begin{enumerate}
\item
$C_{\ker \gvp_1}(\Wt{\gvp}(\ker\gvp))=1$;

\item
$Z(\ker \gvp_1)=1$.
\end{enumerate}
\end{enumerate}
\end{lemma}
\begin{proof}
(1):\quad
$(\gt, g)\in N(\gvp)$ iff $\gvp(\gc)^g=\gvp(\gt(\gc)),$ for all $\gc\in\gC$.
Hence $\gvp(\gc)^g\le\gvp(\gC),$ for all $\gc\in\gC,$ so $g\in N_G(\gvp(\gC))$.
\medskip

\noindent
(2):\quad
If $\gc\in\ker\gvp,$ then $1=\gvp(\gc)^g=\gvp(\gt(\gc)),$ so $\gt(\gc)\in\ker\gvp$. 
Since by Lemma \ref{lem Ngvp is a gp}(2)  $(\gt^{-1}, g^{-1})\in N(\gvp),$
part (2) holds.  
\medskip

\noindent
(3):\quad
By definition, $\Wt\gvp(\gC)=\{(c_{\gc},\gvp(\gc))\mid\gc\in\gC\}$.
Thus if $(\gt, g)\in C_{N(\gvp)}(\Wt\gvp(\gC)),$ then $\gt\in C_{\Aut(\gC)}({\rm Inn}(\gC))$ and $g\in C_G(\gvp(\gC))$.
It is easy to check that $C_{\Aut(\gC)}({\rm Inn}(\gC))=C_{\Aut(\gC)}(\gC/Z(\gC))$.
\medskip

\noindent
(4):\quad
If $(\gt,1)\in\ker p_{\gvp},$ then $\gvp(\gc)=\gvp(\gc)^1=\gvp(\gt(\gc)),$ so we see that
this is equivalent to $\gt$ centralizing $\gC/\ker\gvp$.  The second part of (4) follows from
$\Wt{\gvp}\circ p_{\gvp}=\gvp$.
\medskip

\noindent
(5):\quad
By definition, $\ker\Wt\gvp=\{\gc\in\gC\mid (c_{\gc},\gvp(\gc))=(id, 1)\}$.  Thus
$\gc\in Z(\gC)\cap\ker\gvp$.
\medskip

\noindent
(6):\quad
By (5), $\ker\Wt{\gvp_1}=Z(N(\gvp))\cap\ker(p_{\gvp})$.  Let $(\gt, g)\in\ker\Wt{\gvp_1}$.
By (4), $g=1$.  By (3) and (4), 
$\gt\in C_{\Aut(\gC)}(\gC/Z(\gC))\cap C_{\Aut(\gC)}(\gC/\ker\gvp)=C_{\Aut(\gC)}(\gC/(Z(\gC)\cap\ker\gvp)=C_{\Aut(\gC)}(\gC)=id$,
because $1=\ker\Wt\gvp=Z(\gC)\cap\ker\gvp$.
\medskip

\noindent
(7):\quad
(a):\quad
Note that $\Wt{\gvp}(\ker\gvp)=\{(c_{\gc}, 1)\mid\gc\in\ker\gvp\}\le\ker p_{\gvp}$.
Let now $(\gt,1)\in C_{\ker p_{\gvp}}(\Wt{\gvp}(\ker\gvp))$.  
By (2), $\gt(\ker\gvp)=\ker\gvp$.
It follows that the restriction of $\gt$ to $\ker\gvp$ centralizes $\ker\gvp/Z(\ker\gvp)$.
Since $Z(\ker\gvp)=1,$ we see that $\gt\in C_{\Aut(\gC)}(\ker\gvp)$.
Thus, by (4),
\begin{equation}\label{eq kerpgvp}
\gt\in C_{\Aut(\gC)}(\gC/\ker\gvp)\cap C_{\Aut(\gC)}(\ker\gvp).
\end{equation}
However,  $C_{\Aut(\gC)}(\gC/\ker\gvp)\cap C_{\Aut(\gC)}(\ker\gvp)\cong Z^1(\gC/\ker\gvp,\ Z(\ker\gvp))$ 
(see, e.g.~\cite[Satz I.4.4]{H}).  Since $Z(\ker\gvp)=1,$ equation \eqref{eq kerpgvp}
implies that $\gt=id$. This shows part (a).
\smallskip

\noindent
(b):\quad
Let $(\gt,1)\in Z(\ker p_{\gvp}),$ then, in particular, $(\gt,1)\in C_{\ker p_{\gvp}}(\Wt{\gvp}(\ker\gvp)),$
so by (a), $\gt= id$.
 \end{proof}
 
As a corollary we get the following

\begin{thm}\label{thm nor tower}
Assume  that $Z(\ker\gvp)=1$ and that $\, \gC$ and $G$ are finite. 
Then the normalizers tower $($see diagram  \ref{eq inj nor tower}$)$ 
terminates after a finite number of steps. 
\end{thm}
\begin{proof}
(1):\quad
Consider the normalizers tower of diagram \eqref{eq inj nor tower}.
Let $K_i=\ker\gvp_i,$ for all non-negative integers $i$.

Note that by induction on $i,$  by our assumption
that $Z(\ker\gvp)=1,$ and by Lemma \ref{lem useful}(7b), we have $Z(\ker\gvp_i)=1,$ for all non-negative integers $i$.
By Lemma \ref{lem useful}(5\&6),
$\Wt{\gvp_i}$ is injective, and by Lemma \ref{lem useful}(4), $\Wt{\gvp_{i}}(K_i)\le K_{i+1},$
for all non-negative integers $i$.

Thus, after appropriate identifications, and using the fact
that $\Wt{\gvp_i}$ is a normal map (so its image is normal
in $\gC^{i+1}$), we have $K_i\nsg K_{i+1},$ for all nonnegative
integers $i$, and furthermore,
$C_{K_{i+1}}(K_i)=1$, by Lemma \ref{lem useful}(7a).

We can now apply Theorem \ref{thm W} as in the proof of Wielandt's
Theorem that the automorphism tower of a group with trivial center
terminates after finitely many steps (see, e.g., \cite[13.5.4]{R}).
Namely, by Theorem \ref{thm W}(2), $C_{K_i}(K_0)=1,$ for all positive 
integers $i,$ and hence, by Theorem \ref{thm W}(1), there
exists $s$ such that 
\[
K_s=K_{s+1}=K_{s+2}\dots\ .
\]
However,
by definition, $\gvp_i(\gC^i)\le \gvp_{i+1}(\gC^{i+1}),$
for all $i\ge 0$.  Hence there exists $t$ such that
\[
\gvp_t(\gC^t)=\gvp_{t+1}(\gC^{t+1})=\gvp_{t+2}(\gC^{t+2})\dots\ .
\]  
It follows that for $m=\max\{t,s\}$ we get that
$\gC^m\cong\gC^{m+1}\cong \gC^{m+2}\dots\ .$
\end{proof}

\appendix

\section{Details of the construction of the free normal closure}\label{app A} 

Recall the definition of the group $F$ from the beginning of \S\ref{sect ex nc}.

We define a map of groups
\[
\W\gvp\colon F\to G:\ \gc_g\mapsto \gvp(\gc)^g,\quad \gc\in\gC,\ g\in G.
\] 

First we show that the map $\W\gvp$ respects the relations $R_1, R_{nm1}$ and $R_{nm2}$.

\begin{lemma}\label{lem hatgvp respects R}
The map $\W\gvp$ respects the relations $R_1, R_{nm1}, R_{nm2}$.
\end{lemma}
\begin{proof}
For the relations $R_1$ we have
\[
\W\gvp(1_g)=\gvp(1)^g=1^g=1\in G,\quad\forall g\in G.
\]

For the relations $R_{nm1},$ let $\gc, \gd\in\gC$ and $g\in G,$ then we get
%
\begin{gather*} 
\W\gvp(\gc_g\gd_g)=\W\gvp(\gc_g)\W\gvp(\gd_{g})=\gvp(\gc)^g\gvp(\gd)^g\\
=\gvp(\gc\gd)^g=\W\gvp((\gc\gd)_g).
\end{gather*}
Hence
\[
\W\gvp(\gc_g\gd_g)=\W\gvp((\gc\gd)_g).
\]
Finally
for the relations $R_{nm2},$ let $\gc,\gd\in\gC$ and $h,g\in G,$ then
%
%
\begin{gather*} 
\W\gvp(\gd_h^{-1}\gc_g\gd_h)=\W\gvp(\gd_h^{-1})\W\gvp(\gc_g)\W\gvp(\gd_h)\\
=(\gvp(\gd)^h)^{-1}\gvp(\gc)^g\gvp(\gd)^h=\gvp(\gc)^{g\gvp(\gd)^h}=\W\gvp(\gc_{g\gvp(\gd)^h}).
\end{gather*}
Hence
\[
\W\gvp(\gd_h^{-1}\gc_g\gd_h)=\W\gvp(\gc_{g\gvp(\gd)^h}),
\]
and the lemma holds.
\end{proof}

By Lemma \ref{lem hatgvp respects R} the map $\W\gvp$
 induces
a map, which we continue to denote by  $\W\gvp$,
as in equation \eqref{eq hatgvp}.

\begin{lemma}\label{lem the aut h}
The map $\W{\ell}\colon G\to\cl$ defined in equation \eqref{eq action on cl}
is a well defined homomorphism of groups.
\end{lemma}
\begin{proof}
We start by defining an action of $G$ on $F$.  
For each $h\in G$ we define an automorphism $\ell(h)\in\Aut(F)$.
We denote the image of $\gc_g\in F$ under $\ell(h)$ by $\gc_g^h,$
and we define:
\[
\gc_g^h=\gc_{gh},\quad\text{for all }\gc\in\gC,\, g\in G.
\]
Since $F$ is a free group on the generators $(\gC\sminus\{1\})_G,$
$\ell(h)$ determines a unique automorphism of $F$.

Next we show that the automorphism $\ell(h)$ of $F$ preserves
the relations $R_{nm1}$ and $R_{nm2}$.

For the relations $R_{nm1}$ we need show that
\[
(\gc_g)^h(\gd_g)^h=((\gc\gd)_g)^h,\quad\forall\gc,\gd\in\gC\text{ and }\forall g\in G,
\]
is a relation in $R_{nm1}$.  Indeed 
\[
(\gc_g)^h(\gd_g)^h=\gc_{gh}\gd_{gh}\quad\text{ while }\quad ((\gc\gd)_g)^h=(\gc\gd)_{gh}.
\]
and
\[
\gc_{gh}\gd_{gh}=(\gc\gd)_{gh},
\]
is a relation in $R_{nm1}$.

For the relations $R_{nm2}$ we need show that
\[
(\gd_t^{-1}\gc_g\gd_t)^h=(\gc_{g\gvp(\gd)^t})^h,\quad\forall\gc,\gd\in\gC\text{ and }\forall h,g,t\in G,
\]
is a relation in $R_{nm2}$.  Indeed 
\[
(\gd_t^{-1}\gc_g\gd_t)^h=\gd_{th}^{-1}\gc_{gh}\gd_{th}\quad\text{while}\quad(\gc_{g\gvp(\gd)^t})^h=\gc_{g\gvp(\gd)^th}.
\]
and
\[
\gd_{th}^{-1}\gc_{gh}\gd_{th}=\gc_{gh\gvp(\gd)^{th}} (=\gc_{g\gvp(\gd)^th}),
\]
is a relation in $R_{nm2}$.

Since $\ell(h)\in\Aut(F)$ preserves the relations $R_{nm1}$ and $R_{nm2},$
we can define the automorphism $\W{\ell}$ of equation \eqref{eq action on cl} 
on the generators $\{\W\gc_g\mid\gc\in\gC\text{ and }g\in G\}$ of $\cl,$
and extend $\W\ell$ to an automorphism of $\cl$.
Since $\W\ell(h)$ has an inverse $\W\ell(h^{-1}),$
it is indeed an automorphism of $\cl$.
Further
$(\W\gc_g)^{hh'}=\W\gc_{ghh'}=(\W\gc_g^h)^{h'},$
Thus
$\W\ell\colon G\to\Aut(\cl)$
is a homomorphism.
\end{proof}

\begin{lemma}\label{lem Wgvp is a nm}
The map $\W\gvp\colon\cl\to G$ is a normal map,
having the normal structure $\W\ell\colon G\to\Aut(\cl)$.
\end{lemma}
\begin{proof}
Notice that we have
\[
((\W\gc_g)^h)\W\gvp=(\W\gc_{gh})\W\gvp=\gvp(\gc)^{gh}=\left(\gvp(\gc)^g\right)^h=((\W\gc_g)\W\gvp)^h,
\]
for all $\gc\in\gC$ and $g, h\in G$. That is
\[
((\W\gc_g)^h)\W\gvp=((\W\gc_g)\W\gvp)^h,\quad\forall g, h\in G.
\]
This shows that (NM1) in the definition of a normal map
is satisfied by $\W\gvp,$ on the generators of $\W\gC$
and hence (NM1) holds for $\W\gvp$.

Next we wish to show that (NM2) in the definition of a normal map
is satisfied by $\W\gvp$.  It suffices to show this for the 
generators of $\W\gC,$ that is, we need to show that 
%
\begin{equation}\tag{nm2}\label{eq nm2!}
\W\gc_g^{(\W\gd_h)\W\gvp}=\W\gc_g^{\W\gd_h},\quad\forall\gc,\gd\in\gC\text{ and }\forall g,h\in G.
\end{equation}
By the definition of $\W\gvp$ equation \eqref{eq nm2!}
means
\[
\W\gc_g^{(\W\gd_h)\W\gvp}=\W\gc_g^{\W\gd_h}\iff \W\gc_g^{\gvp(\gd)^h}=\W\gd_h^{-1}\W\gc_g\W\gd_h
\iff \W\gc_{g\gvp(\gd)^h}=\W\gd_h^{-1}\W\gc_g\W\gd_h.
\]
Since the last equality holds in $\cl,$ the lemma holds.
\end{proof}

\setcounter{subsection}{3}
\subsection{\bf The universality of $\cl$}\label{sub univ of nc}\hfill
\medskip

\noindent
We first define
\[
c_{\gvp}\colon\gC\to\cl:\ \gc\mapsto\W\gc_1.
\]
Let now $n\colon M\to G$ be a normal map such that
there exists $\psi\colon\gC\to M$ with $\psi\circ n=\gvp$.
We define a map
\[
\W\psi\colon F\to M:\ \gc_g\mapsto \psi(\gc)^g,
\]
where recall that $M$ is a crossed module, so there is an action
of $G$ on $M$ and $\psi(\gc)^g$ is the image of $\psi(\gc)$ under the
action of $g$.  Now
\[
(\gc_g\gd_g)\W\psi=\psi(\gc)^g\psi(\gd)^g=\psi(\gc\gd)^g=((\gc\gd)_g)\W\psi.
\]
Further
\begin{align*}
(\gc_{g\gvp(\gd)^h})\W\psi&=\psi(\gc)^{gh^{-1}\gvp(\gd)h}\\
&=\psi(\gc)^{gh^{-1}((\gd)\psi n)h}\quad\text{(because $\gvp=\psi\circ n$)}\\
&=\psi(\gc)^{gh^{-1}(\psi(\gd))h}\quad \text{(because $n$ is a normal map)}\\
&=\psi(\gc)^{g(\psi(\gd))^h}.
\end{align*}
while
\begin{align*}
(\gd_h^{-1}\gc_g\gd_h)\W\psi&=((\gd_h)\W\psi)^{-1}\cdot((\gc_g)\W\psi)\cdot(\gd_h)\W\psi\\
&=\left(\psi(\gd)^h\right)^{-1}\psi(\gc)^gc(\gd)^h\\
&=\psi(\gc)^{(g\psi(\gd))^h}.
\end{align*}
It follows that $\W\psi$ induces a homomorphism which we also denote by $\W\psi:$
\[
\W\psi\colon\cl\to M:\ \gc_g\mapsto \psi(\gc)^g.
\]
We have
$(\gc) c_{\gvp}\circ \W\psi=(\W\gc_1)\W\psi=\psi(\gc)^1=\psi(\gc)$.
Thus
$c_{\gvp}\circ \W\psi =\psi$. 
Similarly 
$c_{\gvp}\circ\W\gvp=\gvp$.
Next let 
\[
\W\mu\colon\cl\to M,\text{ with }c_{\gvp}\circ \W\mu=\psi,
\]
where $\W\mu$ is a normal morphism.  Then
$(\gc)c_{\gvp}\circ \W\mu=(\W\gc_1)\W\mu=(\gc)\psi$.
Also
\[
\W\mu(\W\gc_g)=\W\mu((\W\gc_1)^g)=(\W\mu(\W\gc_1))^g=\psi(\gc)^g,\quad\forall\gc\in\gC\text{ and }\forall g\in G.
\]
Since $\cl$ is generated by $\{\W\gc_g\},$ we conclude that $\W\mu=\W\psi,$ and $\W\psi$ is unique.

\setcounter{prop}{4}
\begin{remark}
Let $f\colon \gC\to \gC'$,  and consider the diagram
\[
\xymatrix{
\gC'\ar[r]^{f}\ar@/_2ex/[rrdd]_{c_{f\circ\gvp}} &\gC\ar[rr]^{\gvp}\ar[rd]_{c_{\gvp}} && G\\
&&\cl\ar[ru]^{\W\gvp}\\
&&(\gC')^{f\circ\gvp}\ar[ruu]_{\W{f\circ\gvp}}\ar[u]^{\W{f\circ c_{\gvp}}}\\
}
\]
Thus we see that $f$ induces a normal morphism $\W{f\circ c_{\gvp}}\colon(\gC')^{f\circ\gvp}\to \cl$.
\end{remark}

\section{The normal closure of an inclusion of a normal subgroup}\label{app B}
\noindent
In this appendix we prove a lemma about $\cl$ in the situation where $\gvp\colon\gC\to G$ is an inclusion
of a normal subgroup.

\begin{lemma}\label{lem normalsgp}
Assume that $\gvp$ is injective and that $\gvp(\gC)\nsg G$.
Then
\begin{enumerate}
\item
$c_{\gvp}(\gC)$ is a retract of $\cl,$ and $\cl=c_{\gvp}(\gC)\times\ker\W\gvp$;

\item
we can take $\cl=\gvp(\gC)\times\ker\W\gvp,$ 
with the action
\[
(\gvp(\gc),v)^g=(\gvp(\gd),uv^g),\text{ where }c_{\gvp}(\gc)^g=c_{\gvp}(\gd)u,\ \gd\in\gC\text{ and }u, v\in\ker\W\gvp,
\]
for all $\gc\in\gC$ and $v\in\ker\W\gvp$.
We take $c_{\gvp}$ to be the map
defined by $\gc\mapsto (\gvp(\gc),0),$  
and $\W\gvp$ the map defined by $(\gvp(\gc),v)\mapsto \gvp(\gc),$
for all $\gc\in\gC$ and $v\in\ker\W\gvp$.  In particular,
if $\cl= c_{\gvp}(\gC),$ then we can take $c_{\gvp} = \gvp$ and consequently $\cl=\gvp(\gC)$ and 
$\W\gvp=inc$ is the inclusion map.  

\item 
The following are equivalent
\begin{enumerate}
\item
$\cl=c_{\gvp}(\gC)$.

\item
$c_{\gvp}(\gC)$ is $G$-invariant in $\cl$.

\item
$\psi(\gC)$ is $G$-invariant in $M,$
for all factorization $\gC\overset{\psi} \lr M\overset{m}\lr G$ of $\gvp,$ with
$m$ a normal map.
\end{enumerate}

\item
The following are equivalent
\begin{enumerate}
\item
$\cl=\gvp(\gC)$.
\item
For all abelian groups $V$ and
any action of $G$ on $\gvp(\gC)\times V$ such that
$(\gvp(\gc), 0)^{\gvp(\gd)}=(\gvp(\gc^{\gd}),0),$ for all $\gc, \gd\in\gC,$ the subgroup
$\{(\gvp(\gc), 0)\mid\gc\in\gC\}$ of $\gvp(\gC)\times V$
is $G$-invariant.
\end{enumerate}
\end{enumerate}
\end{lemma}
\begin{proof}
(1):\quad
Since by Lemma \ref{lem gen of WgC}(2) $(\cl)\W\gvp$ is
the normal closure of $(\gC)\gvp$ is $G,$ we see that $(\cl)\W\gvp=(\gC)\gvp,$
so the map
$\W\gvp\circ\gvp^{-1}\circ c_{\gvp}$ is a well defined map on $\cl$
and $((\gc) c_{\gvp})\W\gvp\circ\gvp^{-1}\circ c_{\gvp}=(\gc)\gvp\circ\gvp^{-1}\circ c_{\gvp}=(\gc)c_{\gvp},$
for all $\gc\in\gC$. This shows that $(\gC)c_{\gvp}$ is a retract of $\cl$.
Let us now compute the kernel of the retraction $\W\gvp\circ\gvp^{-1}\circ c_{\gvp}$.
Since $\gvp$ (and hence $c_{\gvp}$) are injective, we see that this kernel is $\ker\W\gvp$.
By Lemma \ref{lem widehatgC}(3), $\ker\W\gvp\le Z(\cl),$
Hence (1) holds.
\medskip

\noindent
(2):\quad
Let $\Wt{\gC}=\gvp(\gC)\times\ker\W\gvp$.
Note that the action of $G$ on $\Wt\gC$ is
defined by
$(\gvp(\gc), v)^g=\gt\left((\gt^{-1}(\gvp(\gc), v))^g\right)$,
where $\gt\colon\cl\to\Wt\gC$ is the isomorphism
defined by $c_{\gvp}(\gc)v\mapsto (\gvp(\gc),v)$ (see part (1)).
Note further that we are using Lemma \ref{lem basic cm}(1),
so that $v^g\in\ker\W\gvp,$ for all $v\in\ker\W\gvp$.
Hence we indeed defined an action of $G$ on $\Wt\gC$.

Let $\Wt{c_{\gvp}}$
be the map defined by $\gc\mapsto (\gvp(\gc),0),$  
and $\Wt{\gvp}$ be the map defined by $(\gvp(\gc),v)\mapsto \gvp(\gc),$
for all $\gc\in\gC$ and $v\in\ker\W\gvp$.

Assume that $\gC\overset{\psi}\lr M\overset{n}\lr G$ is a factorization
of $\gvp$.  Let $\W\psi\colon\cl\to M$ be the unique map of diagram \eqref{diag fnc}, and set
\[
\Wt{\psi}\colon\Wt\gC\to M:\quad (\gvp(\gc), v)\mapsto\W\psi(c_{\gvp}(\gc)v)=\psi(\gc)\W\psi(v).
\]
Let $(\gvp(\gc),v)\in\Wt\gC,$ let $g\in G$ and write $c_{\gvp}(\gc)^g=c_{\gvp}(\gd)u,$
with $u\in\ker\W\gvp$. 
Then
\begin{gather*}
\Wt{\psi}((\gvp(\gc), v)^g) 
=\Wt{\psi} (\gvp(\gd),uv^g)=\W\psi(c_{\gvp}(\gd)uv^g)=
\W\psi(c_{\gvp}(\gc)^g)\W\psi(v^g)=(\W\psi(c_{\gvp}(\gc))^g\W\psi(v)^g\\
=(\psi(\gc)\W\psi(v))^g=(\Wt{\psi}(\gvp(\gc),v))^g.
\end{gather*}
This shows that $\Wt{c_{\gvp}}$ is an initial normal morphism over $G$.

Also,
\begin{gather*}
\Wt\gvp((\gvp(\gc),v)^g)=\Wt\gvp((\gvp(\gd),uv^g)=\gvp(\gd)=(\gd)(c_{\gvp}\circ\W\gvp)=\W\gvp(c_{\gvp}(\gd))\\
=\W\gvp(c_{\gvp}(\gc)^gu^{-1})=\W\gvp(c_{\gvp}(\gc)^g)=\W\gvp(c_{\gvp}(\gc))^g=\gvp(\gc)^g\\
=(\Wt\gvp(\gvp(\gc),v))^g.
\end{gather*}
So $\Wt\gvp$ satisfies (NM1).  Further,
\begin{equation}\label{eq nm2(1)}
c_{\gvp}(\gc)^{\gvp(\gd)}=c_{\gvp}(\gc)^{\W\gvp(c_{\gvp}(\gd)}=c_{\gvp}(\gc)^{c_{\gvp}(\gd)}=c_{\gvp}(\gc^{\gd}),\quad\forall \gc,\gd\in\gC.
\end{equation}
Notice also that for $v\in\ker\W\gvp$ and $\gd\in\gC$ we have 
\begin{equation}\label{eq nm2(2)}
v^{\gvp(\gd)}=v^{\W\gvp(c_{\gvp}(\gd)}=v^{c_{\gvp}(\gd)}=v,
\end{equation}
because $v\in Z(\cl)$.  Hence, by the definition of $\Wt\gvp$ and of the action of
$G$ on $\Wt\gC$ we get using equation \eqref{eq nm2(1)} and \eqref{eq nm2(2)},
\[
(\gvp(\gc), v)^{(\gvp(\gd), u)\Wt\gvp}=(\gvp(\gc), v)^{\gvp(\gd)}=(\gvp(\gc^{\gd}),v^{\gvp(\gc)})=(\gvp(\gc^{\gd}),v)=
(\gvp(\gc), v)^{(\gvp(\gd), u)},
\]
and we see that $\Wt\gvp$ satisfies (NM2) as well.
Thus $\Wt\gvp$ is a normal map.
Next
\[
(\gc)\Wt{c_{\gvp}}\circ\Wt\psi=(\gvp(\gc),0)\Wt\psi=(\gc)\gvp,
\]
and 
\[
(\gvp(\gc),v)\Wt\psi\circ n=n(\psi(\gc)\W\psi(v))=n(\psi(\gc))n(\W\psi(v))=\gvp(\gc)\W\gvp(v)=\gvp(\gc),
\]
for all $\gc\in\gC$ and $v\in\ker\W\gvp$.

It follows that we can replace in diagram \eqref{diag fnc}
$\cl, c_{\gvp}, \W\gvp$ with $\Wt{\gC}, \Wt{c_{\gvp}}, \Wt\gvp$ respectively,
and then the map $\Wt\psi$ renders the diagram commutative.
The proof of the uniqueness of $\Wt\psi$ is  similarly, so (2) holds. 
\medskip

\noindent
(3):\quad
(a)$\Leftrightarrow$ (b):
If $\cl=c_{\gvp}(\gC),$ then of course $c_{\gvp}(\gC)$ is $G$-invariant, while if  $c_{\gvp}(\gC)$
is $G$-invariant, then by Lemma \ref{lem gen of WgC}(1), $\cl=c_{\gvp}(\gC)$.   
\smallskip

\noindent
(b)$\Leftrightarrow$ (c):
Assume that $\psi(\gC)$ is $G$-invariant for every factorization as in (c).
Then, in particular, taking $\psi=c_{\gvp}, M=\cl$ and $n=\W\gvp,$ (c)
implies that $\cl=c_{\gvp}(\gC)$.  Conversely, assume that $\cl=c_{\gvp}(\gC)$.
Then by the commutativity of diagram \eqref{diag fnc}, and since $\W\psi$ is a normal morphism
over $G,$ we have that $(\cl)\W\psi$ is $G$-invariant in $M$.
That is, $(\gC)c_{\gvp}\circ\W\psi=(\gC)\psi$ is $G$-invariant in $M,$ and (c) holds.
\medskip

\noindent
(4):\quad 
(a)$\Rightarrow$(b):  Assume (a) holds.
Let $M:=\gvp(\gC)\times V$ be as in (b).
Let $\psi\colon \gC\to M$ and $n\colon M\to G$ be defined by
$\psi(\gc)=(\gvp(\gc),0),$ and $n(\gvp(\gc),v)=\gvp(\gc),$ for all
$\gc\in\gC$ and $v\in V$.  It is easy to check that $n$ is a normal
map and clearly $\psi\circ n=\gvp$.   By the equivalence of
(a) and (c) in (2), the subgroup $\{(\gvp(\gc), 0)\mid\gc\in\gC\}$ must be
$G$-invariant.
\smallskip

\noindent
(b)$\Rightarrow$(a):  Assume that (b) holds.  Suppose
$\cl\ne\gvp(\gC)$.  By (2) we may assume that $\cl=\gvp(\gC)\times V,$ with $V=\ker\W\gvp,$
so $V$ is abelian (see Lemma \ref{lem basic cm}(1)).  
Let $(\gvp(\gc),v)\in\cl$ and let $\gvp(\ga)\in\gvp(\gC)$.
Note that $c_{\gvp}(\gc)^{\gvp(\ga)}=c_{\gvp}(\gc)^{\W\gvp(c_{\gvp}(\ga))}=c_{\gvp}(\gc^{\ga}),$
because $\W\gvp$ is a normal map.
By the definition of the action of $G$ on $\cl$ we have
\[
(\gvp(\gc),0)^{\gvp(\ga)}=(\gvp(\gc^{\ga}),0),
\]
for all $\ga,\gc\in\gC$.
Further,
by Lemma \ref{lem gen of WgC}(1), $\cl$ is generated by $c_{\gvp}(\gC)=\{(\gvp(\gc),0)\mid \gc\in\gC\}$.
Thus $M=\cl$ violates (b). 
\end{proof}

\begin{remark}
We do not know how to fully characterize those groups $\gC$ and $G$
that satisfy the property in 4(b) of Lemma \ref{lem normalsgp}.
However, if $\gC$ is a perfect group, then clearly 
this property holds.
\end{remark}

\subsection*{Acknowledgment.}
We thank the referee for the careful reading of the paper
and for various useful remarks.


\end{document}